\documentclass[12pt,oneside]{amsart}

\usepackage[utf8]{inputenc}
\usepackage[english]{babel}
\usepackage[margin=3cm]{geometry} 
\usepackage{amsmath,amsthm,amssymb}
\usepackage{mathtools}
\usepackage{graphicx}
\usepackage[colorlinks=true, allcolors=blue]{hyperref}
\usepackage{mathrsfs,calligra}
\usepackage{tikz-cd,adjustbox}
\usepackage{spverbatim,comment}
\usepackage[style=alphabetic,backend=biber]{biblatex}
\usepackage[shortlabels]{enumitem}

\usepackage{fancyhdr}

\theoremstyle{plain}
\newtheorem{theorem}{Theorem}[section]
\newtheorem{lemma}[theorem]{Lemma}
\newtheorem{prop}[theorem]{Proposition}
\newtheorem{corollary}[theorem]{Corollary}

\newtheorem{fact}[theorem]{Fact}

\newtheorem{obs}[theorem]{Observation}

\theoremstyle{definition}
\newtheorem{definition}[theorem]{Definition}

\theoremstyle{remark}
\newtheorem{remark}[theorem]{Remark}

\newcommand{\Rom}[1]{\uppercase\expandafter{\romannumeral#1}}

\newcommand{\h}{\mathfrak h}

\newcommand{\mc}{\mathcal}
\newcommand{\ms}{\mathscr}

\newcommand{\C}{\mathcal C}

\newcommand{\M}{\mathcal M}

\renewcommand{\AA}{\mathbb A}
\newcommand{\CC}{\mathbb C}
\newcommand{\DD}{\mathbb D}

\newcommand{\GG}{\mathbb G}

\newcommand{\QQ}{\mathbb Q}

\newcommand{\ZZ}{\mathbb Z}

\DeclareMathOperator{\Gal}{Gal}

\DeclareMathOperator{\Tr}{tr}

\newcommand{\SL}{\operatorname{SL}}
\newcommand{\GL}{\operatorname{GL}}

\DeclareMathOperator{\PGL}{PGL}

\DeclareMathOperator{\im}{Im}

\DeclareMathOperator{\Hom}{Hom}
\DeclareMathOperator{\sHom}{\mathscr{H}\text{\kern -3pt {\calligra\large om}}\,}

\DeclareMathOperator{\Pic}{Pic}

\DeclareMathOperator{\Prym}{Prym}
\DeclareMathOperator{\Sp}{Sp}

\DeclareMathOperator{\CH}{CH}
\DeclareMathOperator{\Id}{Id}
\DeclareMathOperator{\Corr}{Corr}
\DeclareMathOperator{\acts}{\curvearrowright}
\DeclareMathOperator{\Nm}{Nm}
\DeclareMathOperator{\Mod}{Mod}

\DeclareFontFamily{U}{mathx}{}
\DeclareFontShape{U}{mathx}{m}{n}{<-> mathx10}{}
\DeclareSymbolFont{mathx}{U}{mathx}{m}{n}
\DeclareMathAccent{\widecheck}{0}{mathx}{"71}

\DeclarePairedDelimiter\abs{\lvert}{\rvert}

\makeatletter
\let\oldabs\abs
\def\abs{\@ifstar{\oldabs}{\oldabs*}}
\makeatother

\theoremstyle{definition} 
\newcommand{\thistheoremname}{}
\newtheorem*{genericthm}{\thistheoremname}

\newcommand{\Addresses}{{
  \vspace{\bigskipamount}
  \footnotesize

  \textsc{Department of Mathematics, MIT, 77 Massachusetts Avenue, Cambridge, MA 02139, USA}\par\nopagebreak
  \textit{E-mail address}: 
  \texttt{annelars@mit.edu}
  
}}

\title{Mapping class group action on the cohomology of the $\SL_n$ character variety}
\author{Anne Larsen}
\date{\today}

\bibliography{main.bib}
\begin{document}

\begin{abstract}
We describe the mapping class group action on the cohomology of the twisted $\SL_n$-character variety of a surface $\Sigma_g$ of genus $g$. Our main tool is a relative version of the endoscopic decomposition of Maulik-Shen \cite{MS}; this allows us to reduce the problem to the mapping class group action on the cohomology of a canonical finite cover of $\Sigma_g$, which was studied by Looijenga \cite{Looijenga}.

\end{abstract}

\maketitle
\section{Introduction}
Fix relatively prime integers $n \in \ZZ_{>0}$ and $d \in \ZZ$, as well as a genus $g \ge 2$. In this paper we examine the ``twisted" character variety of $\SL_n$-local systems on a once-punctured genus $g$ surface $\Sigma_g \setminus p$ with monodromy $e^{2\pi i d/n} \Id_n$ on the loop $\gamma_p$ around the puncture, i.e., the quotient
$$M^{\SL} := \{\phi \in \Hom(\pi_1(\Sigma_g \setminus p), \SL_n(\CC)): \phi(\gamma_p) = e^{2\pi i d/n} \Id_n\}/\PGL_n(\CC)$$
$$\cong \{A_1, B_1, \cdots, A_g, B_g \in \SL_n(\CC): [A_1, B_1] \cdots [A_g, B_g] = e^{2\pi i d/n} \Id_n\}/\PGL_n(\CC),$$
where the adjoint group $\PGL_n(\CC)$ acts on the target $\SL_n(\CC)$ by conjugation. 
The mapping class group $\Mod(\Sigma_g \setminus p)$ acts naturally on this space by pullback of local systems, i.e., by its outer action on $\pi_1(\Sigma_g \setminus p)$, and our goal is to describe the induced action on $H^*(M^{\SL})$.

In order to state the theorem, we recall that subgroups $H_1, H_2$ of a group $G$ are said to be \textit{commensurable} if $H_1 \cap H_2$ is of finite index in each $H_i$. Also, let $\widetilde{\Sigma}_g^n \to \Sigma_g$ be the finite cover corresponding to the canonical surjection $\pi_1(\Sigma_g) \twoheadrightarrow H_1(\Sigma_g, \ZZ/n)$.

\begin{theorem} \label{thm: intro}
The kernel of the $\Mod(\Sigma_g \setminus p)$ action on $H^*(M^{\SL}, \ZZ)$ is commensurable to the kernel of the $\Mod(\Sigma_g \setminus p)$ action on $H^1(\widetilde{\Sigma}_g^n, \ZZ)$.
\end{theorem}
By work of Looijenga on the mapping class group action on cohomology of finite abelian covers of $\Sigma_g$ \cite{Looijenga}, this then leads to the following concrete description of the algebraic monodromy group (i.e., the identity component of the Zariski closure of the image of $\Mod(\Sigma_g \setminus p)$ in $\GL(H^*(M^{\SL}, \CC))$):
\begin{corollary} \label{cor: intro}
If $g > 2$, or if $g = 2$ and $(n,6) = 1$, then the algebraic monodromy group of $\Mod(\Sigma_g \setminus p)$ acting on $H^*(M^{\SL}, \CC)$ is the commutator subgroup of $\Sp_{\Gal}(H^1(\widetilde{\Sigma}_g^n, \CC))$, the subgroup of invertible linear transformations respecting both the intersection form and the Galois action of $H_1(\Sigma_g, \ZZ/n)$ on $\widetilde{\Sigma}_g^n$.
\end{corollary}

Although the discussion until now has been purely topological, our proof goes via the nonabelian Hodge correspondence, which provides a diffeomorphism between $M^{\SL}$ and the moduli space of stable $\SL_n$-Higgs bundles of degree $d$ on any complex algebraic curve $C$ of genus $g$. In general, moduli spaces of Higgs bundles tend to be more amenable to algebro-geometric methods, and there are results on the cohomology of character varieties for which the only known proof uses the Higgs bundle description in an essential way. 

For example, if we consider the analogously defined twisted character variety $M^{\GL}$, work of Markman on moduli of stable sheaves on surfaces, as applied to moduli spaces of Higgs bundles \cite[Theorem 3]{Markman-integral}, implies the following:
\begin{fact} \label{fact: GL}
    The kernel of the $\Mod(\Sigma_g \setminus p)$ action on $H^*(M^{\GL}, \ZZ)$ is the kernel of the surjection
    $\Mod(\Sigma_g \setminus p) \twoheadrightarrow \Mod(\Sigma_g) \twoheadrightarrow \Sp(H^1(\Sigma_g, \ZZ)).$
    In particular, the monodromy group is $\Sp_{2g}(\ZZ)$.
\end{fact}
The proof runs as follows:
there is a set of ``tautological" classes in $H^*(M^{\GL})$ defined by taking K\"unneth components of Chern classes of the universal local system on $(\Sigma_g \setminus p) \times M^{\GL}$, and Markman's result is that these tautological classes generate $H^*(M^{\GL})$.
Since the mapping class group action preserves the universal local system and hence its Chern classes, an immediate corollary is a description of the $\Mod(\Sigma_g \setminus p)$ action on $H^*(M^{\GL})$ in terms of the action on $H^*(\Sigma_g)$. But now it is a classical fact that the mapping class group action on $H^*(\Sigma_g)$ factors through a surjection to $\Sp(H^1(\Sigma_g, \ZZ))$. (For a more careful exposition, see, e.g., \cite[\S 4.1]{SP}.)

Returning to the case of $\SL$, Theorem \ref{thm: intro} shows that the mapping class group action is now no longer controlled by $H^1(\Sigma_g)$ but rather $H^1(\widetilde{\Sigma}_g^n)$, reflecting the fact that the cohomology $H^*(M^{\SL})$ is not generated by tautological classes. The issue is the following:
\begin{obs}
    The finite group $\Hom(\pi_1(C \setminus p), \mu_n) \cong H^1(C, \ZZ/n) =: \Gamma$ acts on $M^{\SL}$ via the action of $\mu_n$ on $\SL_n$ by scalar multiplication.
\end{obs}
\noindent On $M^{\GL}$, this extends to an action of the connected group $\Hom(\pi_1(C \setminus p), \CC^*) \cong (\CC^*)^{2g}$, and so the corresponding action on $H^*(M^{\GL})$ is trivial. In particular, since the tautological classes on $H^*(M^{\SL})$ are pulled back along the inclusion $M^{\SL} \subset M^{\GL}$, they belong to $H^*(M^{\SL})^\Gamma$.

The existence of $\Gamma$-variant classes in $H^*(M^{\SL})$ was first pointed out by Hausel and Thaddeus \cite{HT}, who interpreted them in terms of stringy cohomology of the corresponding $\PGL$ space in the context of SYZ mirror symmetry. In general, these variant classes are not so well understood; unlike the tautological classes, there is no obvious geometric construction, and so, for example, we do not fully understand the Hodge structure on the variant part of $H^*(M^{\SL})$, although some information on dimensions can be recovered via point-counting techniques \cite{Mereb}. However, more recently Maulik and Shen \cite{MS} were able to relate the variant classes of the moduli space of $\SL$ Higgs bundles on a curve $C$ to the cohomology of moduli spaces of $\GL_{n/m}$-Higgs bundles on $m$-fold covers of $C$. The main technical tool of this paper is a relative (i.e., compatible with the monodromy or mapping class group action) version of their construction, which occupies \S \ref{sec: relative}, and the implications for the mapping class group action on $H^*(M^{\SL})$, along the lines of the proof of Fact \ref{fact: GL} above, are described in \S \ref{sec: monodromy}.

\textbf{Acknowledgements}: I thank my advisor, Davesh Maulik, for suggesting the problem and for helpful conversations. Thanks also to Mirko Mauri for a useful discussion of the relative nonabelian Hodge correspondence. This work was partially supported by an NSF Graduate Research Fellowship under grant no.~2141064.

\section{Relativizing the Maulik-Shen construction} \label{sec: relative}
In this section we prove that the endoscopic correspondence of \cite[Theorem 0.4]{MS}, which is a surjection of cohomology groups of Higgs moduli spaces for a single curve, can be upgraded to a surjection of local systems over a cover of $\M_{g,1}$.

\subsection{Set-up and notation} \label{subsec: setup}
In this section we will work over the base space $\M := \M_{g,1} \times_{\M_g} \M_g[n]$, where by $\M_g[n]$ we mean the moduli space of curves with level $n$ structure introduced in \cite[(5.14)]{DM}. (We note that $\M$ is scheme if $n \ge 3$, and the reader may assume this is always the case by either restricting to the dense open schematic locus or using $\M_g[4]$ when $n = 2$.) Let $\C := \M \times_{\M_g} \M_{g,1}$ be the universal curve, and $\pi: \C \to \M$ the projection map. 
Also, let $s: \M \to \C$ be the section pulled back from the canonical section of $\M_{g,1} \times_{\M_g} \M_{g,1} \to \M_{g,1}$.

\begin{prop}
    There is a subgroup $S$ of the group of $n$-torsion line bundles on $\C$ such that the restriction to any fiber
    $\Pic(\C)[n] \to \Pic([C] \times_\M \C)[n]$
    defines an isomorphism $S \xrightarrow{\sim} \Pic(C)[n]$.
\end{prop}

\begin{proof}
    The projection $\M \to \M_g[n]$ gives, by definition, a canonical isomorphism $R^1\pi_*(\ZZ/n\ZZ) \cong (\ZZ/n\ZZ)^{2g}$ of \'etale sheaves with alternating form.
    Consider the morphism
    \begin{equation} \label{eq: pic torsion}
        (\ZZ/n\ZZ)^{2g} = H^0(\M, R^1\pi_*(\ZZ/n\ZZ)) \to H^0(\M, R^1\pi_* \GG_m) = \Pic_{\M/\C}(\M)
    \end{equation}
    induced by the inclusion $\ZZ/n\ZZ \to \GG_m$.
    Since $\pi: \C \to \M$ has geometrically integral fibers, using the section $s$ described above, we identify $\Pic_{\M/\C}(\M)$ with the subspace of $\Pic(\C)$ consisting of line bundles $\mc L$ such that $s^* \mc L \in \Pic(\M)$ is trivial \cite[Tag 0D28]{stacks-project}.
    Under this identification, for each $\gamma \in (\ZZ/n\ZZ)^{2g}$ we take $L_\gamma \in \Pic(\C)$ to be the image of $\gamma$ under (\ref{eq: pic torsion}). Note that by construction, $L_\gamma$ is an $n$-torsion class. Finally, for each $i: p \hookrightarrow \M$, by proper base change we have $i^* R\pi_*(\ZZ/n\ZZ) = R\pi_{p *} (\ZZ/n\ZZ)$. In particular, we get a morphism from (\ref{eq: pic torsion}) to the corresponding diagram for $\pi_p: \C_p \to p$, which is the inclusion $\Pic^0(\C_p)[n] \to \Pic^0(\C_p)$ \cite[Tag 03RQ]{stacks-project}.
    Thus each $L_\gamma$ restricts to the corresponding element of $\Pic^0(\C_p)[n]$.
\end{proof}

For the rest of this section, we fix an element $\gamma \in (\ZZ/n\ZZ)^{2g}$ of order $m$, and let $r := n/m$. We define $\pi_\gamma: \C_\gamma \to \C$ to be the $m$-fold \'etale cover corresponding to the $m$-torsion line bundle $L_\gamma$, i.e., $V(t^m - 1) \subset |L_\gamma|$, where $t$ is the tautological section on the pullback of $L_\gamma$ to the total space $|L_\gamma|$. We let $G_\gamma \cong \ZZ/m\ZZ$ denote the Galois group, which acts by scaling the fiber coordinate on $|L_\gamma|$ by elements of $\mu_m$. On each fiber $C$, the above gives a canonical identification $(\ZZ/n\ZZ)^{2g} = \Pic^0(C)[n] = H^1(C, \ZZ/n) =: \Gamma$, and the covering space $\pi_\gamma$ corresponds to a canonical map $H_1(C, \ZZ/n) \to G_\gamma$ whose kernel consists of elements pairing trivially with $\gamma$. Also, via the intersection pairing on $H^1(C, \ZZ/n)$, the choice of $\gamma \in \Gamma$ gives a canonically defined element $\kappa := \langle \gamma, - \rangle \in \hat{\Gamma}$. 

Now we are ready to define the various moduli spaces of Higgs bundles arising in the proof.

\begin{itemize}
    \item Let $M_{n,d}(\C) \to \M$ be the (coarse) moduli space of fiberwise stable ($\GL_n$) Higgs bundles of rank $n$ and relative degree $d$ on $\C_\gamma/\M$; this was constructed by Simpson via GIT in \cite[Theorem 4.7]{Simpson-reps1}.
    \item Let
    $$\AA(\C,n) := |\pi_* \omega_{\C/\M}| \times_{\M} |\pi_* \omega^{\otimes 2}_{\C/\M}| \times_{\M} \cdots \times_{\M} |\pi_* \omega_{\C/\M}^{\otimes n}|$$
    be the relative Hitchin base and
    $$h: M_{n,d}(\C) \to \AA(\C,n) \to \M$$
    the (proper) relative Hitchin morphism. We also take $\AA_0(\C,n) \subset \AA(\C,n)$ to be given by the inclusion of the zero section $\M \xhookrightarrow{0} |\pi_* \omega_{\C/\M}|$.
    \item Given an element $\ms L \in \Pic_{\C/\M}(\M)$, we define the coarse moduli space of stable $\SL_n$ Higgs bundles as follows: consider
    $$M^0_{n,d}(\C) := M_{n,d}(\C) \times_{h, \AA(\C,n), 0} \AA_0(\C,n)$$
    the moduli space of trace-free stable Higgs bundles. Then
    $$\widecheck{M}_{n, \ms L}(\C) := M_{n,d}^0(\C) \times_{\det,\Pic_{\C/\M}, \ms L} \M,$$
    where the map $M_{n, d}^0(\C) \to \Pic_{\C/\M}$ is given by $(E, \theta) \mapsto \det E$. For the purposes of this paper, we fix $\ms L := \ms O_{\C}(ds(\M))$, where $s(\M)$ is the image of the section described above.
    \item Given $\gamma \in \Gamma$, tensor product by $L_\gamma$ defines an endomorphism of $\widecheck{M}_{n, \ms L}(\C)$. Let $M_\gamma(\C)$ be the fixed locus, and let $\AA_\gamma(\C) \subset \AA_0(\C,n)$ be the image of $M_\gamma(\C)$ under $h$.
    \item Given $\gamma \in \Gamma$, we define a space
    $$M_\pi(\C_\gamma) := \M_{r,d}(\C_\gamma) \times_{\M_{n,d}(\C)} \widecheck{M}_{n, \ms L}(\C),$$
    where the first map is given by pushforward
    $$\M_{r,d}(\C_\gamma) \mapsto \M_{n,d}(\C), (E, \theta) \mapsto (\pi_{\gamma *} E, \pi_{\gamma *} \theta)$$
    and the second by the inclusion $\widecheck{M}_{n, \ms L}(\C) \subset M_{n,d}(\C)$. We also define $\AA_\pi(\C) \subset \AA(\C_\gamma, r)$ to be the image of $M_\pi(\C_\gamma) \subset M_{r,d}(\C_\gamma)$ under the relative Hitchin map.
\end{itemize}
\begin{remark} \label{remark: D defs}
    In \S \ref{subsec: Yun} we will also use the following generalization of the definitions above:
    if $\ms D \in \Pic(\C)$ is of relative degree $>2g-2$ on the fibers of $\C/\M$ or $\ms D = \omega_{\C/\M}$, we define $M_{n,d}^{\ms D}(\C), \AA^{\ms D}(\C,n)$, etc. as before, with $\omega_{\C/\M}$ replaced by $\ms D$ everywhere. (That is, the Higgs field is now a morphism $\theta: E \to E \otimes \ms D$, and $\AA^{\ms D}(\C,n)$ is changed accordingly.) Similarly, we define $M_{n,d}^{\ms D}(\C_\gamma)$, etc. as before but with $\omega_{\C_\gamma/\M}$ replaced by $\pi_\gamma^* \ms D$.
\end{remark}

\begin{remark}
    We defined all these spaces for $\C/\M$, but for any morphism $S \to \M$, we define corresponding spaces $M_{n,d}(\C_S)$, etc. by base change. In particular, for $* \to \M$ we recover the correspondingly defined spaces of \cite{MS}. When we restrict to a fiber, we use plain font $C := \C_*$, $L := \ms L|_{\C_*}$, etc. This more or less recovers the notation of \cite{MS}, with the following changes: we use $\pi_\gamma: \C_\gamma \to \C$ instead of $\pi: C' \to C$, with $\pi$ now denoting all maps to the base $\M$, and the $\GL/\SL$ variants are written $M/ \widecheck{M}$ by the conventions of, e.g., \cite{PW}, rather than the $\widetilde{M}/M$ of \cite{MS}.
\end{remark}

\subsection{Preparatory steps}
We begin by describing the relative versions of the more or less routine steps of the \cite{MS} construction.
In what follows, we use $\pi$ to denote all maps to the base $\M$ (or suitable covers) when there seems to be no risk of confusion. We also use $\QQ(X)$ to denote the constant sheaf on $X$ and $D^b(X, \QQ)$ for the bounded derived category of constructible sheaves of $\QQ$-vector spaces on $X$.
\begin{prop} \label{prop: step 1}
We have
$$R\pi_* \QQ(M_{r,d}(\C_\gamma)) \cong R\pi_* \QQ(M_{1,0}(\C)) \otimes (R\pi_* \QQ(M_\pi(\C_\gamma)))^\Gamma$$
in $D^b(\M, \QQ)$.
\end{prop}

\begin{proof}
We start by noting that tensor product with $L_\delta$ for $\delta \in \Gamma$ (respectively, $\pi_\gamma^* L_\delta$) gives an action of $\Gamma$ on $M_{1,0}(\C)$ (respectively, $M_\pi(\C_\gamma)$), fiberwise with respect to the projection maps to $\M$. Moreover, as in the fiberwise case, the product map
$$M_{1,0}(\C) \times_{\M} M_\pi(\C_\gamma) \to M_{r,d}(\C_\gamma), \,\,((L, s), (E, \theta)) \mapsto (\pi_\gamma^* L \otimes E, \pi_\gamma^* s \otimes 1 + 1 \otimes \theta)$$
is a quotient by the diagonal action of $\Gamma$. Thus, we have
\begin{equation} \label{eq: step 1-1}
R\pi_* \QQ(M_{r,d}(\C_\gamma)) \cong (R\pi_*\QQ (M_{1,0}(\C) \times_{\M} M_\pi(C_\gamma))^\Gamma.
\end{equation}
Next, we note that $M_{1,0}(\C) = \Pic^0_{\C/\M} \times_\M |\pi_* \omega_{\C/\M}|$ is topologically locally trivial over $\M$, and so $R^i\pi_* \QQ(M_{1,0}(\C))$ is a local system for each $i$. Since $\M$ and the fibers $M_{1,0}(C) \cong \Pic^0(C) \times H^0(C, \omega_C)$ and $M_\pi(C)$ are smooth \cite[Proposition 1.1]{MS}, by smooth relative K\"unneth (see, e.g., \cite[Lemma 3.16]{SP}) we have
\begin{equation} \label{eq: step 1-2}
R\pi_*\QQ (M_{1,0}(\C) \times_{\M} M_\pi(\C_\gamma) \cong R\pi_* \QQ(M_{1,0}(\C)) \otimes R\pi_*\QQ(M_\pi(\C_\gamma)).
\end{equation}
Now, we note that $R\pi_* \QQ(M_{1,0}(\C))$ is $\Gamma$-invariant: indeed, by the local triviality, it suffices to see that fiberwise, the $\Gamma$ action extends to an action of the connected group $\Pic^0(C)$ and thus is trivial on cohomology. Taking $\Gamma$ invariants of both sides of (\ref{eq: step 1-2}) in combination with (\ref{eq: step 1-1}), the statement follows.
\end{proof}

\begin{prop} \label{prop: project}
We have
$$R\pi_* \QQ(M_{1,0}(\C)) \cong \oplus_{i=0}^{2g} \bigwedge\nolimits^i R^1\pi_* \QQ(M_{1,0}(\C))[-i]$$
in $D^b(\M, \QQ)$.
In particular, there is a well-defined, split map
$$R\pi_* \QQ(M_{1,0}(\C)) \otimes (R\pi_* \QQ(M_\pi(\C_\gamma)))^\Gamma \to (R\pi_* \QQ(M_\pi(\C_\gamma)))^\Gamma.$$
\end{prop}

\begin{proof}
Using the description of $M_{1,0}(\C)$ above, we see that
$$R\pi_* \QQ(M_{1,0}(\C)) = R\pi_* \QQ(\Pic^0_{\C/\M}),$$
where $\Pic^0_{\C/\M} \to \M$ is an abelian scheme of dimension $g$.
\end{proof}

Recall that the space $M_\pi(\C_\gamma)$ is equipped with commuting actions of $\Gamma$ (by tensor product with the line bundles $\pi_\gamma^* L_\delta$) and $G_\gamma$ (by pullback under deck transformations, which preserve the line bundles $\pi_\gamma^*L_\delta$ since these are pulled back from $\C$).
Since both groups act fiberwise with respect to the projection $\pi: M_\pi(\C_\gamma) \to \M$, we obtain a decomposition of $R\pi_* \CC(M_\pi(\C_\gamma))$ into characters of $\Gamma \times G_\gamma$.
We begin with a description of the action of $\Gamma$ and $G_\gamma$ on the components of $M_\pi(C_\gamma)$. (Recall that $\kappa \in \hat{\Gamma}$ was defined to be dual to $\gamma \in \Gamma$ under the intersection pairing.)

\begin{lemma} \label{lemma: Ggamma = Gamma/K}
    Fiberwise,
    there is a canonical morphism
    \begin{equation} \label{eq: Gamma pi0}
    \Gamma \twoheadrightarrow \Gamma/\ker(\kappa) \xrightarrow{\sim} \pi_0(\Prym(C_\gamma/C)), \delta \mapsto [\pi_\gamma^* L_\delta^{\otimes r}].
    \end{equation}
    Given the choice of $d$ relatively prime to $n$, there is also a canonical isomorphism
    \begin{equation} \label{eq: Ggamma pi0}
    G_\gamma \xrightarrow{\sim} \pi_0(\Prym(C_\gamma/C)), \xi \mapsto \im(\Pic^d(C_\gamma) \mapsto \Prym(C_\gamma/C), A \mapsto A^{-1} \otimes \xi^* A).    
    \end{equation}
     Moreover, letting $L$ be the restriction of $\ms L$ to $C$, both $\Gamma$ and $G_\gamma$ act on $\pi_0((\det \pi_{\gamma *})^{-1}(L))$ via the map
    $$\pi_0(\Prym(C_\gamma/C)) \acts \pi_0((\det \pi_{\gamma *})^{-1}(L))$$
    induced by the natural $\Prym(C_\gamma/C)$-torsor structure of $(\det \pi_{\gamma*})^{-1}(L) \subset \Pic^d(C_\gamma)$.
\end{lemma}

\begin{proof}
    The description of morphism (\ref{eq: Gamma pi0}) comes from \cite[Proposition 1.3]{MS}. By definition, we say that $\Gamma$ acts on $(\det \pi_{\gamma*})^{-1}(L)$ by tensor product with $\pi_\gamma^* L_\delta^{\otimes r}$, and so the compatibility with $\Gamma \to \pi_0(\Prym(C_\gamma/C))$ is clear. (Also, the description of $(\det \pi_{\gamma *})^{-1}(L)$ as a torsor follows from \cite[Proposition 3.10]{HP}.)
    
    As for $G_\gamma$, we note first that morphism (\ref{eq: Ggamma pi0}) is well-defined because $\Pic^d(C_\gamma)$ is connected. We also note that the $G_\gamma$ action on $\pi_0((\det \pi_{\gamma *})^{-1}(L))$ does factor as described: indeed, choosing any $A \in (\det \pi_{\gamma *})^{-1}(L)$, we have
    $$\xi \cdot A = A \otimes (A^{-1} \otimes \xi^* A)$$
    where by definition $[A^{-1} \otimes \xi^* A] = \im(\xi) \in \pi_0(\Prym(C_\gamma/C))$. Finally, since the image of (\ref{eq: Ggamma pi0}) in $\pi_0(\Prym(C_\gamma/C)) \cong \ZZ/m$ acts transitively on the set $\pi_0((\det \pi_{\gamma *})^{-1}(L))$ of cardinality $m$ \cite[Facts on p.~20]{HT}, the morphism (\ref{eq: Ggamma pi0}) must be surjective, thus (again for cardinality reasons) an isomorphism.
    
\end{proof}
In particular, using the identification $\Gamma/\ker(\kappa) \cong G_\gamma$ given by the lemma (which depends nontrivially on $d$ but is otherwise canonical), we identify the character of $\Gamma/\ker(\kappa)$ induced by $\kappa \in \hat{\Gamma}$ with a character $\kappa_{G} \in \widehat{G_\gamma}$.

\begin{prop} \label{prop: st = kappa}
After passing to a finite cover $S \to \M$, there is an automorphism of $R\pi_* \CC(M_\pi(C_\gamma))$ inducing isomorphisms
$$(R\pi_* \CC(M_\pi(\C_\gamma)))_{(\alpha, \beta)} \cong (R\pi_* \CC(M_\pi(\C_\gamma)))_{(\alpha + \kappa,\beta + \kappa_G)}$$
in $D^b(S, \CC)$, for each $(\alpha, \beta) \in \hat{\Gamma} \times \hat{G}_\gamma$.
\end{prop}

\begin{proof}
Given a smooth family $X \to \M$, we let $\pi_{0, rel}(X) \to \M$ denote the cover whose fiber over $m \in M$ is $\pi_0(X_m)$. As described in \cite[Proof of Proposition 1.1]{MS}, the morphism
$$q: M_{r,d}(\C_\gamma) \to M_{1,d}(\C_\gamma), (E, \theta) \mapsto (\det E, \Tr \theta)$$
restricted to $M_\pi(\C_\gamma)$
induces a homeomorphism $\pi_{0, rel}(M_\pi(\C_\gamma)) \cong \pi_{0, rel}(\im(q))$, where $\im(q)$ is an affine bundle over $\Nm^{-1}(\ms L) := \Pic^d(\C_\gamma/\M) \times_{(\det \pi_{\gamma *}), \Pic^d(\C/\M), \ms L} \M$.
In particular $\pi_{0,rel}(\im(q)) = \pi_{0, rel}(\Nm^{-1}(\ms L))$, where the latter is a torsor of the relative group $\pi_{0, rel}(\Prym(\C_\gamma/\C)) = (\Gamma/\ker(\kappa)) \times \M$ by Lemma \ref{lemma: Ggamma = Gamma/K}. After passing to a finite cover $S \to \M$ (say, of order $\le m = |\Gamma/\ker(\kappa)|$), we may assume that this torsor is trivial, so that $M_\pi(\C_{\gamma, S})$ splits into $m$ components.
Furthermore, after arbitrarily choosing a preferred component $M_0$ to trivialize the torsor, we may identify the set of components with the group $\Gamma/\ker(\kappa) \cong G_\gamma$. (Note that the morphism $q$ and the projection $\im(q) \to \Nm^{-1}(\ms L)$ are $\Gamma$- and $G_\gamma$-equivariant, where $\Gamma$ acts on $M_{1,d}(\C_\gamma)$ and $\Nm^{-1}(\ms L)$ by $\pi_\gamma^* L_\delta^{\otimes r}$, as described in Lemma \ref{lemma: Ggamma = Gamma/K}. In particular, the identification $\pi_{0, rel}(M_\pi(\C_\gamma)) \cong \pi_{0, rel}(\Nm^{-1}(\ms L))$ is $\Gamma \times G_\gamma$-equivariant.)

We now consider the automorphism $f$ of
$$R\pi_* \CC(M_{\pi}(\C_{\gamma, S})) = \oplus_{\beta \in \Gamma/\ker(\kappa)} R\pi_* \CC(\beta \cdot M_0) \in D^b(S, \CC)$$
given by scaling each $R\pi_* \CC(\beta \cdot M_0)$ by $\kappa(\beta) \in \CC^*$. To see that this induces the needed $\Gamma$-character shift, choose $\eta \in \Gamma$ whose image generates the cyclic group $\Gamma/\ker(\kappa)$. Then for each $0 \le i < m$ we take $\eta^i: M_0 \to [\eta^i] \cdot M_0$ as our preferred isomorphism, inducing an identification
$$R\pi_* \CC(M_\pi(\C_{\gamma, S})) \cong \oplus_{i=0}^{m-1} R\pi_* \CC(M_0) \in D^b(S, \CC).$$
Given a character $\alpha \in \hat{\Gamma}$, under this identification we have
$$ (R\pi_* \CC(M_\pi(\C_{\gamma, S})))_\alpha \cong (1, \alpha(\eta), \ldots, \alpha(\eta)^{m-1}) \cdot (R\pi_* \CC(M_0))_{\alpha|_{\ker(\kappa)}} \subset \oplus_{i=0}^{m-1} R\pi_* \CC(M_0)$$
as $\ker(\kappa)$ acts trivially on each component, $\eta$ permutes the components as described, and these commute and generate $\Gamma$. Since $f$ acts on the $i^{\mathrm{th}}$ factor $R\pi_* \CC(M_0)$ by $\kappa(\eta)^i$, it is clear from this formula that there is a character shift $\alpha \to \alpha + \kappa$. Finally, the proof for the $G_\gamma$-character shift is the same, except that there is no kernel of the action on $\pi_0(M_\pi(\C_{\gamma,S}))$ to consider, and that via the identifications of Lemma \ref{lemma: Ggamma = Gamma/K}, now $f$ acts according to the character $\kappa_G$ of $G_\gamma$.
\end{proof}

\begin{prop} \label{prop: project Ggamma}
    There is a split projection map
    $$(R\pi_* \CC(M_\pi(\C_\gamma))_\kappa \to (R\pi_*\CC(M_\pi(\C_\gamma)))_\kappa^{G_\gamma}$$
    in $D^b(\M, \CC)$.
\end{prop}

\begin{proof}
    As described previously, the actions of $\Gamma$ and $G_\gamma$ on $M_\pi(\C_\gamma)$ commute, and so the isotypic component $(R\pi_* \CC(M_\pi(\C_\gamma)))_\kappa$ is equipped with an action of (the finite group) $G_\gamma$. 
\end{proof}

\begin{prop} \label{prop: Gpi quotient}
We have an isomorphism
$$(R\pi_* \CC(M_\pi(\C_\gamma)))^{G_\gamma}_\kappa \cong (R\pi_* \CC(M_\gamma(\C)))_\kappa$$
in $D^b(\M, \CC)$.
\end{prop}

\begin{proof}
It follows from the fiberwise result \cite[(29)]{MS} that the pushforward map
$$f: M_\pi(\C_\gamma) \to M_\gamma(\C), (E, \theta) \mapsto (\pi_{\gamma *} E, \pi_{\gamma *} \theta)$$
is a quotient map by the (free) action of $G_\gamma$. Also note that $f$ is $\Gamma$-equivariant by the projection formula. Then we have $(Rf_* \CC(M_\pi(\C_\gamma)))^{G_\gamma} = \CC(M_\gamma(\C))$ as $\Gamma$-equivariant sheaves, and so pushing forward to $\M$ and taking $\kappa$-isotypic components, we recover the isomorphism
$$(R\pi_* \CC(M_\gamma(\C_\gamma)))^{G_\gamma}_\kappa = R\pi_* (Rf_*  \CC(M_\pi(\C_\gamma)))^{G_\gamma})_\kappa \cong (R\pi_* \CC(M_\gamma(\C)))_\kappa.$$
\end{proof}

\subsection{The Ng\^{o}-Yun correspondence} \label{subsec: Yun}
In this section we construct the relative version of the Ng\^{o}-Yun isomorphism for $\ms D$-Higgs bundles, when $\ms D$ is of relative degree $> 2g-2$ and even.

We recall the definitions of moduli spaces of $\ms D$-Higgs in Remark \ref{remark: D defs}. We also define the following:
\begin{itemize}
    \item Let $\ms C_{spec}^{\ms D}(\C,n) \subset |\ms D| \times_\M \AA^\ms D(\C, n)$ be the universal spectral curve defined, as usual, by the polynomial in the tautological section on $|\ms D|$ with coefficients given by $\AA^{\ms D}(\C,n)$. Let $p_\C: \ms C_{spec}^{\ms D}(\C,n) \to \C \times_{\M} \AA^{\ms D}(\C,n)$ be the finite map induced by the projection $|\ms D| \to \C$.
    \item Recall that the Hitchin morphism $M^{\ms D}_{n,d}(\C) \to \AA^{\ms D}(\C,n)$ identifies $M^{\ms D}_{n,d}(\C)$ as a compactified relative Jacobian of the family $\ms C^{\ms D}_{spec}(\C,n) \to \AA^{\ms D}(\C,n)$. We use $M^{\ms D}_{n,d}(\C)^\circ$ to denote the open subvariety parametrizing line bundles on the spectral curves.
    \item There is a $G_\gamma$ action on $\AA^{\ms D}(\C_\gamma, r)$ described as follows: the action of $G_\gamma$ on $\C_\gamma$ induces an action on $(\pi_{\gamma *} \ms O_{\C_\gamma}) \otimes \ms D \cong  \pi_{\gamma *} \pi_\gamma^* \ms D$, therefore on $\pi_* (\pi_\gamma^* \ms D)$, therefore on $\AA^{\ms D}(\C_\gamma, r)$.
    \item Let $\AA^\heartsuit_\pi(\C_\gamma) \subset \AA^{\ms D}(\C_\gamma, r)$ be the intersection of the locus where
    $\ms C_{spec}^{\ms D}(\C_\gamma, r) \to \AA^{\ms D}(\C_\gamma,r)$ is smooth and the locus where $G_\gamma$ acts freely. (This is a Zariski open subset, whose restriction to a fiber is the open subset $\AA^\heartsuit(\pi)$ of \cite[\S 3.2]{MS}.)
    \item By the fiberwise result \cite[Lemma 5.1]{HP}, the quotient $\AA_\pi^{\ms D}(\C_\gamma)/G_\gamma$ injects into $\AA_0(\C, n)$ with image $\AA_\gamma(\C)$, compatibly with the pushforward $M_\pi(\C_\gamma) \to M_\gamma(\C)$, which is also a $G_\gamma$ quotient map.
    Let $\AA^{\heartsuit}_\gamma(\C)$ be the image of $\AA^{\heartsuit}_\pi(\C_\gamma)$ under this map.
    \item The quotient map $\AA_\pi^{\ms D}(\C_\gamma) \to \AA_\gamma^\ms D(\C)$ has the following interpretation in terms of spectral curves \cite[\S 5.1]{HP}: Identifying $a \in \AA_\pi^{D}(C_\gamma)$ with the section $s_a \in H^0(C_\gamma, \pi_\gamma^* D^{\otimes r})$ whose zero set is the corresponding spectral curve, we have 
    $$a \mapsto \prod_{g \in G_\gamma} g^* s_a \in H^0(C_\gamma, \pi_\gamma^* D^{\otimes n})^{G_\gamma} = H^0(C, D^{\otimes n}).$$
    In particular, the $G_\gamma \times G_\gamma$-quotient map
    $$\AA_\pi^{\ms D}(\C_\gamma) \times_\M |\pi^*_\gamma \ms D| \to \AA_\gamma^\ms D(\C) \times_\M |\ms D|$$
    induces a morphism
    \begin{equation} \label{eq: normalization}
       \ms C_{spec}^\ms D(\C_\gamma, r)|_{\AA_\pi^\ms D(\C_\gamma)} \to \ms C^\ms D_{spec}(\C, n)|_{\AA_\gamma^\ms D(\C)} 
    \end{equation}
    over $\AA_\gamma^\ms D(\C)$.
    \item For the purposes of the next proposition, we define
    $$\ms C^\heartsuit_\pi := \ms C^{\ms D}_{spec}(\C_\gamma, r)|_{\AA_\pi^\heartsuit(\C_\gamma)} \,\,\,\,\textrm{and}\,\,\,\, \ms C^\heartsuit_\gamma := \ms C_{spec}^\ms D(\C, n)|_{\AA^\heartsuit_\gamma(\C)}$$
    and consider the natural morphisms
    $$p_\pi: \ms C^\heartsuit_\pi \to \C_\gamma \times_\M \AA_\pi^\heartsuit(\C_\gamma) \to \C \times_\M \AA_\pi^\heartsuit(\C_\gamma) \,\,\,\,\textrm{and}$$
    $$p_\gamma: \ms C^\heartsuit_\gamma \times_{\AA^\heartsuit_\gamma(\C)} \AA^\heartsuit_\pi(\C_\gamma) \to (\C \times_\M \AA^\heartsuit_\gamma(\C)) \times_{\AA^\heartsuit_\gamma(\C)} \AA^\heartsuit_\pi(\C_\gamma) = \C \times_\M \AA^\heartsuit_\pi(\C_\gamma).$$
\end{itemize}

\begin{prop} \label{prop: N}
    There is a line bundle $\ms N \in \Pic(\C)$ pulling back to
    \begin{equation} \label{eq: N}
        \det(p_{\gamma *} \ms O) \otimes \det(p_{\pi *} \ms O)^{-1} \in \Pic(\C \times_\M \AA^\heartsuit_\pi(\C_\gamma)).
    \end{equation}
\end{prop}

\begin{proof}
    Since $p_\gamma$ and $p_\pi$ are both finite, flat, proper maps, the sheaves $p_{\gamma *} \ms O$ and $p_{\pi* }  \ms O$ are vector bundles on $\C \times_\M \AA^\heartsuit_\pi(\C_\gamma)$ 
    \cite[Theorem 25.1.6]{Vakil}, and so (\ref{eq: N}) indeed defines an element of $\Pic(\C \times_\M \AA^\heartsuit_\pi(\C_\gamma))$. Now it suffices to note that $\C \times_\M \AA^{\heartsuit}_\pi(\C_\gamma)$ is an open subset of $\C \times_\M \AA_\pi(\C_\gamma)$, which is an affine bundle over $\C$, and so we obtain a surjection of Chow groups $\CH^1(\C) \to \CH^1(\C \times_\M \AA^\heartsuit_\pi(\C_\gamma))$ \cite[Propositions 1.8, 1.9]{Fulton}. Since all spaces involved are smooth, this is equivalently a surjection $\Pic(\C) \to \Pic(\C \times_\M \AA^\heartsuit_\pi(\C_\gamma))$.
\end{proof}

\begin{prop} \label{prop: g equivariant}
There is a $G_\gamma$- and $\Gamma$-equivariant morphism
$$ g: \widecheck{M}_{n, \ms L}^{\ms D}(\C)^\circ \times_{\AA_0^{\ms D}(\C, n)} \AA^\heartsuit_\pi(\C_\gamma) \to M_{\pi, \ms L \otimes \ms N}^{\ms D}(\C_\gamma)^\circ $$
over $\AA^\heartsuit_\pi(\C_\gamma)$,
restricting over each point of $\M$ to the morphism of \cite[Corollary 3.7]{MS}.
\end{prop}

\begin{proof}
First, we note that by \cite[(18)]{MS} the space $\AA^\heartsuit_\pi(\C_\gamma)$ is independent of the choice of $\ms L$ or $\ms L \otimes \ms N$; the same then holds for the spectral curves.
Identifying the source and target of $g$ as subvarieties of
$$\Pic_{\ms C_\gamma^\heartsuit/\AA_\gamma^\heartsuit(\C)} \times_{\AA_\gamma^{\heartsuit}(\C)} \AA^\heartsuit_\pi(\C_\gamma)
\cong \Pic_{\ms C_\gamma^\heartsuit \times_{\AA_\gamma^\heartsuit(\C)} \AA^\heartsuit_\pi(\C_\gamma)/\AA^\heartsuit_\pi(\C_\gamma)} \,\,\,\,\textrm{and} \,\,\,\,\Pic_{\ms C^\heartsuit_{\pi}/\AA^\heartsuit_{\pi}(\C_\gamma)}$$
respectively, we define $g$ as follows: Consider the pullback map on $\Pic_{-/\AA^\heartsuit_\pi(\C_\gamma)}$ induced by the morphism
$$\ms C^\heartsuit_\pi \to \ms C_\gamma^\heartsuit \times_{\AA^\heartsuit_\gamma(\C)} \AA_\pi^\heartsuit(\C_\gamma)$$
given by (\ref{eq: normalization}) on the first factor and the projection to the Hitchin base on the second. By the fiberwise result \cite[Corollary 3.7, Lemma 3.9]{MS}, the restriction of this morphism to the source of $g$ gives a $G_\gamma$- and $\Gamma$-equivariant morphism with the correct target.
\end{proof}

Consider the graph
$$ \textrm{Graph}(g) \subset (\widecheck{M}_{n, \ms L}^{\ms D}(\C)^\circ \times_{\AA_0^{\ms D}(\C, n)} \AA^\heartsuit_\pi(\C_\gamma)) \times_{\AA^\heartsuit_\pi(\C_\gamma)} M_{\pi, \ms L \otimes \ms N}^{\ms D}(\C_\gamma)^\circ.$$
We will define a cohomological correspondence using its closure
$$\Sigma_\M := \overline{\textrm{Graph}(g)} \subset (\widecheck{M}_{n, \ms L}^{\ms D}(\C) \times_{\AA_0^{\ms D}(\C, n)} \AA^{\ms D}_\pi(\C_\gamma)) \times_{\AA^\ms D_\pi(\C_\gamma)} M_{\pi, \ms L \otimes \ms N}^{\ms D}(\C_\gamma).$$

\begin{prop}
Let $f: X \to Y$ be a finite type morphism of schemes and $S \subset X$ a constructible set. Then there exists a dense open subset $U \subset Y$ such that $\overline{S} \times_Y y = \overline{S \times_Y y}$ for $y \in U$.
\end{prop}

\begin{proof}
Because $S$ is constructible, it contains some dense open subset $S'$ of $\overline{S}$.
Also, if $\eta$ is the generic point of an irreducible component of $Y$ and $s$ the generic point of an irreducible component of $(\overline{S})_{\eta}$, then $s$ must in fact be a generic point of $\overline{S}$, thus contained in $S'$.
Now, by \cite[Tag 054X]{stacks-project} applied to each irreducible component of $Y$, there is a dense open subset $U \subset Y$ such that $S'_y \subset S_y$ is dense in $(\overline{S})_y$ for all $y \in U$, and so we have $\overline{S_y} = (\overline{S})_y$.
\end{proof}

\begin{corollary} \label{cor: restrict Sigma}
There exists an open subset $U \subset \M$ such that over each point of $U$, $\Sigma_\M$ restricts to the $\Sigma$ defined in \cite[(75)]{MS}.
\end{corollary}

Next, we check that the cohomological correspondence induced by a Borel-Moore fundamental class of $\Sigma_\M$ agrees fiberwise with the cohomological correspondence induced by $\Sigma$ in \cite[Theorem 3.10]{MS}.

\begin{prop} \label{prop: fund classes}
Consider the diagram of morphisms of complex varieties
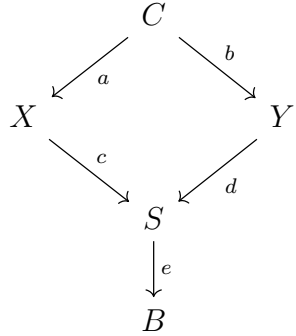
\[\begin{tikzcd}
	& C \\
	X && Y \\
	& S \\
	& B
	\arrow[from=1-2, to=2-1, "a"]
	\arrow[from=1-2, to=2-3, "b"]
	\arrow[from=2-1, to=3-2, "c"]
	\arrow[from=2-3, to=3-2, "d"]
	\arrow[from=3-2, to=4-2, "e"]
\end{tikzcd}\]
where we assume that $C \to X \times_S Y$ and $c, d$ are proper; $e \circ c$ and $B$ are smooth; and $C$ is flat over $B$. Consider a morphism $\alpha: B' \to B$, where $B'$ is smooth, and let $X', Y', C', S'$ be the corresponding base changes, with maps $\alpha_X: X' \to X$, etc. Then the pullback by $\alpha_S^*$ of the morphism in $\Hom(Rd_* \QQ, Rc_* \QQ[2\dim X - 2\dim C])$ induced by a Borel-Moore
fundamental class of $C$ is induced by a fundamental class of $C'$.
\end{prop}

\begin{proof}
First, unraveling definitions, we have
$$[C] \in H^{BM}_{2 \dim C}(C) :=
\Hom(\QQ_C, \DD_C[-2\dim C]) $$
$$= \Hom(b^* \QQ_Y, a^! \DD_X[-2\dim C]) =: \Corr(C; \QQ_X[2(\dim X - \dim C)], \QQ_Y)$$
and similarly for $C'$, using that $X$ and $X'$ are smooth by the assumptions. (Here, we assume $X$ is irreducible; otherwise apply the argument to each connected component separately, with $\dim X$ varying as necessary.) We note that $\dim X - \dim C = \dim X' - \dim C'$ by the flatness assumption. Now, by \cite[Lemma A.4.1]{Yun} it suffices to check that the pullback
$$\alpha^*: \Corr(C; \QQ_X[2 (\dim X -  \dim C)], \QQ_Y) \to \Corr(C'; \QQ_{X'}[2 (\dim X - \dim C)], \QQ_{Y'})$$
takes a fundamental class of $C$ to a fundamental class of $C'$.

Recall that by definition, a morphism $\QQ_C \to \DD_C[-2\dim C]$ is a fundamental class if it is an isomorphism on the nonsingular locus $C^\circ$ of $C$ (or on any open subset dense among the components of top dimension). By definition \cite[\S A.4]{Yun} we see that
$$\alpha^* [C]: (b')^* \QQ_{Y'} =  \alpha_C^* b^* \QQ_Y \xrightarrow{\alpha_C^* [C]} \alpha_C^* a^{!} \QQ_X[2(\dim X - \dim C)]$$
$$\xrightarrow{*! \to !*} (a')^! \alpha_X^* \QQ_X[2(\dim X - \dim C)] = (a')^! \QQ_{X'}[(2(\dim X' - \dim C')]$$
is an isomorphism over $\alpha_C^{-1}(C^\circ) \cap (C')^\circ$, which contains the open subset on which $\pi_B': C' \to B'$ is smooth (which is the preimage of the subset on which $\pi_B$ is smooth, by flatness of $\pi_B$ \cite[Tag 02V4]{stacks-project}). Finally, by generic smoothness of
each component of $C' \to B'$, we see that this subset is dense in each component, as needed.
\end{proof}

\begin{corollary} \label{cor: fund class}
The restriction of the cohomological correspondence induced by a fundamental class of $\Sigma_\M$ to a general fiber is induced by a fundamental class of the restriction of $\Sigma_\M$ to the fiber.
\end{corollary}

\begin{proof}
Let $B \subset \M$ be a dense open subset over which $\Sigma$ is flat,
$S := \AA_\pi^{\ms D}(\C_{\gamma, B})$, $X := M_{\pi, \ms L \otimes \ms N}(\C_{\gamma, B})$, $Y := \widecheck{M}_{n, \ms L}^{\ms D}(\C_B) \times_{\AA_0(\C_B,n)} \AA_\pi^{\ms D}(\C_{\gamma, B})$, $C := \Sigma_B$, and $B' = * \hookrightarrow B$. By construction $C \subset X \times_S Y$ is a closed subset, and $c,d$ are derived from Hitchin morphisms, therefore proper; also $B \subset \M$ is smooth, and $M_{\pi, \ms L \otimes \ms N}(\C_{\gamma, B}) \to B$ has smooth fibers \cite[p. 9]{MS}. Now we apply Proposition \ref{prop: fund classes}.
\end{proof}

Putting these together, we get:
\begin{prop} \label{prop: relative correspondence}
Assume $\ms D$ is of relative degree $> 2g-2$. Restricting to a dense open subset $U \subset \M$, there is an isomorphism
$$(Rh_* \CC(\widecheck{M}_{n, \ms L}^\ms D(\C_U)))_\kappa \cong (Rh_* \CC(M_{\gamma, \ms L \otimes \ms N}^{\ms D}(\C_U)))_\kappa[-2d_\gamma^\ms D] \in D^b(\AA_0^{\ms D}(\C_U, n), \CC),$$
where we use $h$ to denote the projections from $\widecheck{M}_{n, \ms L}^\ms D(\C_U)$ and $M_{\gamma, \ms L \otimes \ms N}^{\ms D}(\C_U)$ to the Hitchin base $\AA_0^\ms D(\C_U, n)$ and $d_\gamma^\ms D := \dim \AA_0^\ms D(\C,n) - \dim \AA_\gamma^\ms D(\C)$.
\end{prop}

\begin{proof}
Let $U \subset \M$ be a dense open set over which Corollaries \ref{cor: restrict Sigma} and \ref{cor: fund class} apply. Then choosing a fundamental class for $\Sigma_U$, we get a $G_\gamma$- and $\Gamma$-equivariant (by Proposition \ref{prop: g equivariant}) morphism
\begin{multline*}
    Rh_* \CC(\widecheck{M}_{n, \ms L}^{\ms D}(\C_U) \times_{\AA_0^{\ms D}(\C_U,n)} \AA_\pi^{\ms D}(\C_{\gamma,U}) ) \to \\
    Rh_* \CC(M_{\pi, \ms L \otimes \ms N}(\C_{\gamma, U}))[2(\dim M_{\pi, \ms L \otimes \ms N}(\C_{\gamma, U}) - \dim \Sigma_U)] \in D^b(\AA_\pi^{\ms D}(\C_{\gamma, U}), \CC),
\end{multline*}
where again we use $h$ to denote both (proper) morphisms to the Hitchin base $\AA_\pi^{\ms D}(\C_{\gamma, U})$.
(Note that the degree shift is $-2d_\gamma^\ms D$ as defined above because all the spaces involved are equidimensional over $U$, and so all codimensions match those in \cite[Theorem 3.10]{MS}.)

Pushing forward along the proper $G_\gamma$ quotient map $q_\AA: \AA_\pi^{\ms D}(\C_{\gamma, U}) \to \AA_\gamma^\ms D(\C_U)$ and taking invariants, we obtain a morphism
\begin{multline} \label{eq: MS iso}
    Rh_* \CC(\widecheck{M}_{n, \ms L}^{\ms D}(\C_U)|_{\AA_\gamma^\ms D(\C_U)}) \otimes (q_{\AA *} q_\AA^*\CC)^{G_\gamma} = Rh_* \CC(\widecheck{M}_{n, \ms L}^\ms D(\C_U)|_{\AA_\gamma^\ms D(\C_U)}) \\
    \to (Rh_{*} q_{M *} \CC(M^\ms D_{\pi, \ms L \otimes \ms N}(\C_{\gamma, U}))[-2d_\gamma^\ms D])^{G_\gamma} = Rh_* \CC(M^{\ms D}_{\gamma, \ms L \otimes \ms N}(\C_U))[-2d_\gamma^\ms D]
\end{multline}
of complexes on $\AA_\gamma^\ms D(\C_U) \subset \AA_0^\ms D(\C_U, n)$,
where in the last equality we use the fact that $q_M: M_{\pi, \ms L \otimes \ms N}^\ms D(\C_{\gamma,U}) \to M_{\gamma, \ms L \otimes \ms N}^\ms D(\C_U)$ is a $G_\gamma$-quotient map \cite[Proposition 7.1]{HT}. Now, the key point is that (\ref{eq: MS iso}) induces an isomorphism on $\kappa$-components: it suffices to check this on fibers, and by proper base change and the assumption on $U$, the fiberwise statement is exactly \cite[Corollary 3.11]{MS}.

To complete the proof, we note that
$(Rh_* \CC(\widecheck{M}_{n, \ms L}^\ms D(\C_U)))_\kappa$ is supported on $\AA_\gamma^\ms D(\C_U)$; since $h$ is proper, this follows from the fiberwise result, see, e.g., \cite[Corollary 2.2]{MS}. (In fact, this corollary gives a much stronger statement about supports; for our purposes, a refined form of the argument of \cite[Theorem 1.4]{HP} would suffice.) In other words, we have
$$(Rh_* \CC(\widecheck{M}^\ms D_{n, \ms L}(\C_U)))_\kappa \cong (Rh_* \CC(\widecheck{M}^\ms D_{n, \ms L}(\C_U)|_{\AA_\gamma^\ms D(\C_U)}))_\kappa.$$
Combining this with the restriction of (\ref{eq: MS iso}) to $\kappa$-components, we obtain the needed isomorphism.
\end{proof}

\begin{prop} \label{prop: change of determinant}
If $\ms D$ is of even degree, then after passing to a finite cover $S$ of $\M$, there is an isomorphism
$$(Rh_* \CC(M_{\gamma, \ms L \otimes \ms N}^\ms D(\C_S)))_\kappa \cong (Rh_* \CC(M_{\gamma, \ms L}^\ms D(\C_S)))_\kappa$$
in $D^b(\AA_\gamma^\ms D(\C_S), \CC)$.
\end{prop}

\begin{proof}
If $\ms D$ is of even relative degree, then $\ms N$ is of relative degree divisible by $n$ \cite[Lemma 3.5]{MS}. In particular, there is a (nonempty) relative $\Pic^0_{\C/\M}[n] \cong \Gamma \times \M$ torsor of $n$th roots of $\ms N$ given by
$$\Pic^{\deg \ms N/n}_{\C/\M} \times_{\otimes^n, \Pic^{\deg \ms N}_{\C/\M}, \ms N} \M.$$
Let $S$ be a finite cover on which this torsor is trivialized and $\ms K$ a global section, considered as an element of $\Pic^{\deg \ms N/n}_{\C_S/S}(S)$. Then tensoring by $\ms K$ defines a $\Gamma$-equivariant isomorphism $\widecheck{M}_{n, \ms L}^\ms D(\C_S) \xrightarrow{\sim} \widecheck{M}_{n, \ms L \otimes \ms N}^\ms D(\C_S)$ over the two projections $h$.
\end{proof}

\subsection{The vanishing cycles trick}
In this section, we explain how to construct the isomorphisms of the last section in the non-$\ms D$-twisted case. Again, we start with definitions.

\begin{definition}
Given a family $X/S$ and line bundles $A, B \in \Pic(X)$, we define a stack $\widecheck{\M}_{n, A}^B(X/S)$ over $S$ as follows: the groupoid corresponding to $T \to S$ has objects $\{(E, \theta), \alpha: \det E \xrightarrow{\sim} A_T\}$, where $E$ is a bundle of rank $n$ on $X_T$ and $\theta: E \to E \otimes B_T$ a linear map with $\Tr \theta = 0$ such that the restriction of $(E, \theta)$ to each geometric fiber is (Gieseker) semistable, with morphisms given by isomorphisms of this data.
\end{definition}
We note the following three special cases:
\begin{enumerate}
    \item Given $\ms D \in \Pic(\C)$, the stack $\widecheck{\M}_{n, \ms L}^\ms D(\C/\M)$ is a $\mu_n$-gerbe over the coarse space $\widecheck{M}_{n, \ms L}^\ms D(\C)$. (Here, we use that $\ms L$ is of relative degree coprime to $n$, and so fiberwise semistable implies fiberwise stable.)
    The same holds for any morphism $T \to \M$.
    \item Recall that there is a section $s: \M \to \C$; we will consider $\widecheck{\M}^\ms D_{n, \ms L}(s(\M)/\M)$, together with the map
    $$\widecheck{ev}_s: \widecheck{\M}_{n, \ms L}^\ms D(\C/\M) \to \widecheck{\M}^\ms D_{n, \ms L}(s(\M)/\M)$$
    given by restricting along $s$.
    (Note that in this case, the fibers are of dimension 0, and so the semistability condition is vacuous.) 
    \item If $U \subset \M$ is an open subset on which $\ms L$ and $\ms D$ are trivialized, then we note that
    $$\widecheck{\M}^\ms D_{n, \ms L}(s(U)/U) \cong \widecheck{\M}^\ms O_{n, \ms O}(s(U)/U) \cong U \times \widecheck{\M}_{n, \ms O}^\ms O(*/*) \cong U \times [\mathfrak{sl}_n/\SL_n].$$
\end{enumerate}
We will also need the following more general definition:
\begin{definition}
    Given a finite morphism of families $X \xrightarrow{p} Y/S$, and given line bundles $A,B \in \Pic(Y)$, we define a stack $\M_{\pi, A}^B(X \to Y/S)$ as follows: the groupoid corresponding to $T \to S$ has objects $\{(E, \theta), \alpha: \det p_*E \xrightarrow{\sim} A_T\}$, where $E$ is a bundle of rank $r$ on $X_T$ and $\theta: E \to E \otimes p^*B_T$ a linear map with $\Tr p_* \theta = 0$ such that the restriction of $(E, \theta)$ to each geometric fiber is (Gieseker) semistable, with morphisms given by isomorphisms of this data.
\end{definition}
Again, we are interested in three cases:
\begin{enumerate}
    \item Given $\ms D \in \Pic(\C)$, the stack $\M_{\pi, \ms L}^\ms D(\C_\gamma \to \C/\M)$ is a $\mu_n$-gerbe over the coarse space $M_{\pi, \ms L}^\ms D(\C_\gamma)$, and similarly for any morphism $T \to \M$.
    \item We will consider the stack $\M_{\pi, \ms L}^\ms D(\pi_\gamma^{-1}(s(\M)) \to s(\M)/\M)$, together with the map
    $$ev^\pi_s: \M_{\pi, \ms L}^\ms D(\C_\gamma \to \C/\M) \to \M_{\pi, \ms L}^\ms D(\pi_\gamma^{-1}(s(\M)) \to s(\M)/\M)$$
    given by restricting to $\pi_\gamma^{-1}(s(\M)).$
    \item If $S \to \M$ is an open subset of a finite cover of $\M$ on which the line bundles $\ms L, \ms D$ and the $m$-fold cover $\pi_{\gamma}^{-1}(s(\M)) \to s(\M)$ are trivialized, then we have
    $$\M_{\pi, \ms L}^{\ms D}(\pi_\gamma^{-1}(s(S)) \to s(S)/S) \cong S \times \M_{\pi, \ms O}^\ms O(\sqcup_{i=1}^m * \to */*) \cong S \times [\h_\pi/H_\pi],$$
    where $\h_\pi \subset \mathfrak{sl}_n$ and $H_\pi \subset \SL_n$ are defined in \cite[\S 4.2]{MS}.
\end{enumerate}
Note that these recover the previous three examples in the case $\gamma = id$.

Given $S \to \M$ as described in case (3), we will consider the morphism
\begin{multline} \label{eq: mu}
    \mu: \M_{\pi, \ms L}^\ms D(\C_{\gamma, S} \to \C_S/S) \xrightarrow{ev^\pi_s} \M_{\pi, \ms L}^{\ms D}(\pi_\gamma^{-1}(s(S)) \to s(S)/S) \\ \cong S \times [\h_\pi/H_\pi] \xrightarrow{p_2} [\h_\pi/H_\pi] \xrightarrow{c_2} \AA^1,
\end{multline}
where $p_2$ is the projection onto the second factor, and $c_2(g)$ is the $t^{n-2}$ coefficient of the characteristic polynomial of $g \in \h_\pi \subset \mathfrak{sl}_n$.

\begin{prop} \label{prop: vanishing cycles}
    Let $\ms D_i :=\omega_{\C/\M} \otimes \ms O_{\C}(is(\M)) $ for $i \in \ZZ_{>0}$.
    Applying the vanishing cycles functor of $\mu$, we get an isomorphism
    $$\varphi_\mu(\QQ(\M_{\pi, \ms L}^{\ms D_i}(\C_{\gamma, S} \to \C_S/S))) \cong \QQ(\M_{\pi, \ms L}^{\ms D_{i-1}}(\C_{\gamma, S} \to \C_S/S))[-\dim \h_\pi]$$
    of complexes in $D^b(\M_{\pi, \ms L}^{\ms D_{i-1}}(\C_{\gamma, S} \to \C_S/S), \QQ).$ This isomorphism is $\Gamma$-equivariant. It is not $G_\gamma$-equivariant but rather shifts characters by $\QQ_-^{\otimes r(m+1)}$, where $\QQ_-$ is the sign representation $G_\gamma \to \{\pm 1\}, \xi \mapsto (-1)^{|\xi|}$.
\end{prop}

\begin{proof}
    First, the map $ev^\pi_s$ is smooth by the fiberwise result \cite[Proposition 4.1]{MS}, as is the projection $p_2: S \times [\h_\pi/H_\pi] \to [\h_\pi/H_\pi]$ since $S$ is smooth over $\M$. Then by smooth base-change for vanishing cycles \cite[Lemma 4.4(c)]{MS}, we have 
    \begin{equation} \label{eq: smooth pullback}
        \varphi_\mu(\QQ(\M_{\pi, \ms L}^{\ms D_i}(\C_{\gamma, S} \to \C_S/S))) \cong (p_2 \circ ev_s^\pi)^* \varphi_{c_2}(\QQ([\h_\pi/H_\pi])),
    \end{equation}
    where
    \begin{equation} \label{eq: vanishing hpi}
        \varphi_{c_2}(\QQ([\h_\pi/H_\pi])) \cong \QQ([0/H_\pi])[-\dim \h_\pi]
    \end{equation}
    by \cite[Lemma 4.3]{MS} and the fact that $H_\pi$ is connected and so acts trivially on the vanishing cycles sheaf. Now, as in \cite[Lemma 4.2]{MS}, note that
    $$\M_{\pi, \ms L}^{\ms D_{i-1}}(\C_{\gamma, S} \to \C_S/S) \subset \M_{\pi, \ms L}^{\ms D_{i}}(\C_{\gamma, S} \to \C_S/S)$$
    is described by the condition that $\theta$ has a zero along $s$, i.e., this is the preimage under $p_2 \circ ev_s^\pi$ of $[0/H_\pi] \subset [\h_\pi/H_\pi]$. Thus, combining (\ref{eq: smooth pullback}) and (\ref{eq: vanishing hpi}), we get
    $$\varphi_\mu(\QQ(\M_{\pi, \ms L}^{\ms D_i}(\C_{\gamma, S} \to \C_S/S))) \cong (p_2 \circ ev_s^\pi)^* \QQ([0/H_\pi])[-\dim \h_\pi]$$
    $$\cong \QQ(\M_{\pi, \ms L}^{\ms D_{i-1}}(\C_{\gamma, S} \to \C_S/S))[-\dim \h_\pi].$$

    For the $\Gamma$-equivariance, we note that the argument above applies equally well to $[\M^{\ms D_i}_{\pi, \ms L}(\C_{\gamma, S} \to \C_S/S)/\Gamma]$ (where $ev^\pi_s$ factors through this quotient since the line bundles $s^* L_\delta$ are trivial).

    For the compatibility with $G_\gamma$, we note that $p_2 \circ ev_s^\pi$ is $G_\gamma$-equivariant, where $G_\gamma$ acts by cyclically permuting the points of $\sqcup_{i=1}^m *$ compatibly with its action on the $m$ components of $\pi_\gamma^{-1}(s(S))$. It therefore suffices to prove the statement for $[\h_\pi/H_\pi] \to \AA^1$, or in fact, after taking another $G_\gamma$-equivariant smooth cover, for $c_2: \h_\pi \to \AA^1$. Following the argument of \cite[Proposition 4.6]{MS}, we consider the decomposition $\h_\pi \cong \oplus_{i=0}^{m-1} \h_i$ by characters of the dual group $\hat{G}_\gamma \cong \ZZ/m$. Since $c_2$ is a $G_\gamma$-invariant quadratic form on $\h_\pi$, it must decompose as a sum of functions $c_2^{i, m-i}$ on $\h_i + \h_{m-i}$, and so by the Thom-Sebastiani theorem, we have
    $$\varphi_{c_2}(\QQ(\h_\pi)) \cong \otimes \varphi_{c_2^{i, m-i}}(\QQ(\h_i + \h_{m-i})).$$
    Moreover, the $G_\gamma$ action on $\varphi_{c_2^{0,m}} \QQ(\h_0)$ is trivial, as the $G_\gamma$ action on $\h_0$ is trivial, and the $G_\gamma$ action on $\varphi_{c_2^{i, m-i}} \QQ(\h_i \oplus \h_{m-i})$ is trivial for $i \ne m-i$ since it can be extended to an action of the connected group $\GG_m$ with weights $(1, -1)$ compatible with the function $c_2^{i, m-i}$.

    It remains to determine the $G_\gamma$ action on $\varphi_{c_2}^{m/2, m/2}(\QQ(\h_{m/2}))$ in the case that $m$ is even. By \cite[Lemma 4.3]{MS}, the form $c_2$ is nondegenerate, and by $G_\gamma$-invariance, the same is true of $c_2^{m/2, m/2}$, i.e., this is of the form $x_1^2 + \ldots + x_{\dim \h_{m/2}}^2$ for some basis $\{x_i\}$ of $\h_{m/2}$.
    Recall that for $x^2: \AA^1 \to \AA^1$ with $G_\gamma$ acting on the source via the sign character, we have
    $$\phi_{x^2}(\QQ(\AA^1)) \cong \QQ(0)[-1] \otimes \QQ_-$$
    since multiplying the coordinate of $\AA^1$ by $-1$ exchanges the two points of a nearby fiber $\{\pm \lambda\}$. Again by Thom-Sebastiani, this implies that
    $$\varphi_{c_2^{m/2,m/2}}(\QQ(\h_{m/2})) \cong \QQ(0)[-\dim \h_{m/2}] \otimes \QQ_-^{\otimes \dim \h_{m/2}}.$$
    Since $\h_\pi \oplus \CC_{triv} \cong V^{\oplus r^2}$, where $V$ is the standard $\CC$-representation of $G_\gamma$, we see that $\dim \h_{m/2} = r^2$. In other words, taking vanishing cycles introduces a nontrivial multiple of the sign character exactly when $m \equiv 0, r \equiv 1$ mod 2. This completes the argument that
    $$\varphi_{\mu}(\QQ(\M^{\ms D_i}_{\pi, \ms L}(\C_{\gamma, S} \to \C_S/S))) \cong \QQ(\M^{\ms D_{i-1}}_{\pi, \ms L}(\C_{\gamma, S} \to \C_S/S)) \otimes \QQ_-^{\otimes r(m+1)}$$
    and now we note that tensoring by any character $\QQ_\chi$ of $G_\gamma$ commutes with taking vanishing cycles.
\end{proof}

\begin{corollary} \label{cor: last step}
After passing to an appropriate Zariski open subset $S$ of a finite cover of $\M$, there is an isomorphism
$$R\pi_* \CC(\widecheck{M}_{n, \ms L}(\C_S))_\kappa \cong R\pi_* \CC(M_\gamma(\C_S))_\kappa[-2d_\gamma]$$
of complexes in $D^b(S, \CC)$, where $d_\gamma$ is defined to be $d_\gamma^{\omega_{\C/\M}}$ from Proposition \ref{prop: relative correspondence}.
\end{corollary}

\begin{proof}
By Propositions \ref{prop: relative correspondence} and \ref{prop: change of determinant}, since $\ms D_2$ is of relative degree $2g$, after passing to a Zariski open subset $S'$ of a finite cover of $\M$, there is an isomorphism
$$(Rh_*\CC(\widecheck{M}^{\ms D_2}_{n, \ms L}(\C_{S'}))_\kappa \cong (Rh_* \CC(M^{\ms D_2}_{\gamma, \ms L \otimes \ms N}(\C_{S'})))_\kappa[-2d_\gamma^{\ms D_2}] \cong (Rh_* \CC(M^{\ms D_2}_{\gamma, \ms L}(\C_{S'})))_\kappa[-2d_\gamma^{\ms D_2}]$$
in $D^b(\AA_\gamma^{\ms D_2}(\C_{S'}), \CC)$. Moreover, the same holds when the coarse spaces $\widecheck{M}$ and $M_\gamma$ are replaced by the $B\mu_n$-gerbes $\widecheck{\M}$ and $\M_\gamma$. In particular, we see that
\begin{equation} \label{eq: pre vanishing}
(Rh_*\CC(\widecheck{\M}^{\ms D_2}_{n, \ms L}(\C_{S'})))_\kappa  \cong (Rh_* \CC(\M^{\ms D_2}_{\pi, \ms L}(\C_{\gamma,S'})))_\kappa^{G_\gamma}[-2d_\gamma^{\ms D_2}].
\end{equation}

Now, by passing to a further open subset of a finite cover, say $S$, we may assume that the finite cover $\pi_\gamma^{-1}(s(S)) \to s(S)$ and the line bundles $s^* \ms L$ and $s^* \ms O_{\C_S}(s(S))$ are trivialized. Then the function $\mu$ of (\ref{eq: mu}) is well-defined, and applying $\varphi_\mu$ twice to each side of (\ref{eq: pre vanishing}), we get the following:

For the left side of the equation, note that we have a $\Gamma$-equivariant isomorphism
$$\varphi_\mu \CC(\widecheck{M}^{\ms D_i}_{n, \ms L}(\C_S)) \cong \CC(\widecheck{M}^{\ms D_{i-1}}_{n, \ms L}(\C_S))[-\dim \mathfrak{sl}_n]$$
by Proposition \ref{prop: vanishing cycles} applied in the case $\gamma = id$. In particular,
$$\varphi_\mu(\varphi_\mu \CC(\widecheck{\M}^{\ms D_2}_{\gamma, \ms L}(\C_S))) \cong \CC(\widecheck{\M}^{\ms D_0}_{\gamma, \ms L}(\C_S))[-2\dim \mathfrak{sl}_n] = \CC(\widecheck{\M}_{\gamma, \ms L}(\C_S))[-2\dim \mathfrak{sl}_n].$$
Since $\mu$ factors through $h$ and vanishing cycles commute with proper pushforward, we obtain
\begin{equation} \label{eq: LHS vanishing cycles}
    \varphi_\mu(\varphi_\mu (Rh_*\CC(\widecheck{\M}^{\ms D_2}_{n, \ms L}(\C_S)))_\kappa)) \cong (Rh_*\CC(\widecheck{\M}_{n, \ms L}(\C_S)))_\kappa[-2\dim \mathfrak{sl}_n].
\end{equation}

On the right side of the equation, we have a $\Gamma$-  and $G_\gamma$-equivariant isomorphism
$$\varphi_\mu(\varphi_\mu(\CC(\M^{\ms D_2}_{\pi, \ms L}(\C_{\gamma, S})))) \cong \CC(\M_{\pi, \ms L}(\C_{\gamma, S}))[-2\dim \h_\pi]$$
by Proposition \ref{prop: vanishing cycles}, where the $G_\gamma$-equivariance follows from the fact that the two factors of $\QQ_-^{\otimes r(m+1)}$ combine to give a trivial character shift. As before, we have
\begin{multline} \label{eq: RHS vanishing cycles}
    \varphi_\mu(\varphi_\mu(Rh_* \CC(\M^{\ms D_2}_{\pi, \ms L}(\C_{\gamma, S})))_\kappa^{G_\gamma}) \cong (Rh_* \CC(\M_{\pi, \ms L}(\C_{\gamma, S})))_\kappa^{G_\gamma}[-2\dim \h_\pi] \\ = (Rh_* \CC(\M_{\gamma, \ms L}(\C_S)))_\kappa[-2\dim \h_\pi].
\end{multline}

Thus, applying $\varphi_\mu$ twice to each side of the isomorphism (\ref{eq: pre vanishing}), from (\ref{eq: LHS vanishing cycles}) and (\ref{eq: RHS vanishing cycles}) we get
$$(Rh_* \CC(\widecheck{\M}_{\gamma, \ms L}(\C_S)))_\kappa \cong (Rh_* \CC(\M_{\gamma, \ms L}(\C_S)))_\kappa[2(\dim \mathfrak{sl}_n - \dim \h_\pi - d_\gamma^{\ms D_2})].$$
The needed statement follows from switching back to the coarse spaces and pushing forward to $S$, since by the fiberwise result \cite[\S 4.2]{MS} we have
$\dim \mathfrak{sl}_n - \dim \h_\pi - d_\gamma^{\ms D_2} = -d_\gamma$.
\end{proof}

\begin{theorem} \label{thm: relative}
There exists a Zariski open subset of a finite cover of $\M$, say $W$, over which there is a morphism with right inverse
$$R\pi_* \CC(M_{r,d}(\C_{\gamma, W})) \rightarrow (R\pi_* \CC(\widecheck{M}_{n, \ms L}(\C_W))_\kappa[2d_\gamma] \in D^b(W, \CC).$$
\end{theorem}

\begin{proof}
Taking $W$ to be the fiber product of the various Zariski open subsets of finite covers appearing in the propositions cited below, we have morphisms with right inverses
$$R\pi_* \CC(M_{r,d}(\C_{\gamma, W})) \overset{\ref{prop: step 1}}{\cong} R\pi_* \CC(M_{1,0}(\C_W)) \otimes (R\pi_* \CC(M_\pi(\C_{\gamma,W})))^\Gamma$$
$$ \overset{\xrightarrow{\ref{prop: project}}}{\leftarrow} (R\pi_* \CC(M_\pi(\C_{\gamma,W})))^\Gamma
\overset{\ref{prop: st = kappa}}{\cong} (R\pi_* \CC(M_\pi(\C_{\gamma, W})))_\kappa$$
$$ \overset{\xrightarrow{\ref{prop: project Ggamma}}}{\leftarrow} (R\pi_* \CC(M_\pi(\C_{\gamma,W})))_\kappa^{G_\gamma} \overset{\ref{prop: Gpi quotient}}{\cong} (R\pi_* \CC(M_\gamma(\C_W)))_\kappa \overset{\ref{cor: last step}}{\cong} (R\pi_* \CC(\widecheck{M}_{n, \ms L}(\C_W)))_\kappa[2d_\gamma].$$
\end{proof}

\section{Algebraic monodromy} \label{sec: monodromy}
Either because of the existence of a smooth compactification with smooth boundary (see \cite[Proposition 3.8.1]{Compactification}) or by a relative form of the nonabelian Hodge correspondence (see Proposition \ref{prop: relative NAH}), the relative moduli spaces of Higgs bundles $M_{r,d}(\C_\gamma)$ and $\widecheck{M}_{n, \ms L}(\C)$ defined in the previous section are topologically locally trivial over the base $\M$ (a finite cover of $\M_{g,1}$). In particular, choosing a point $x \in \M$ corresponding to a curve $\C_x =: C$ with cover $\C_{\gamma, x} =: C_\gamma$, it makes sense to discuss the monodromy representations of $\pi_1(\M,x)=: \pi_1(\M)$ on $H^*(M_{r,d}(C_\gamma))$ and $H^*(\widecheck{M}_{n,L}(C))$.

We use $\ker(\pi_1(\M) \acts H^*(M_{r,d}(C_\gamma)))$, etc. to mean the kernel of the monodromy representation of $\pi_1(\M)$ on $\oplus_i R^i\pi_* \CC(M_{r,d}(C_\gamma))$. (Note that the kernel of the monodromy representation is independent of the choice of $\CC$ or $\QQ$ coefficients, and so we suppress the field from the notation.)

\begin{lemma} \label{lemma: GL}
For each $\gamma \in \Gamma$, we have
$$\ker(\pi_1(\M) \acts H^*(M_{r,d}(C_\gamma))) = \ker(\pi_1(\M) \acts H^1(C_\gamma)).$$
\end{lemma}

\begin{proof}
We recall several well-known facts, referring to \cite[\S 3-4]{SP} for proofs or references. First of all, we have
$$R\pi_* \QQ(M_{r,d}(\C_\gamma)) \cong R\pi_* \QQ(M^{\GL}_{1,0}(\C_\gamma)) \otimes R\pi_* \QQ(\widehat{M}_{r,d}(\C_\gamma))$$
\cite[(13), (17)]{SP}, where $\widehat{M}_{r,d}(\C_\gamma) := M_{r,d}(\C_\gamma)/M_{1,0}(\C_\gamma)$ is a relative moduli space of $\PGL$-Higgs bundles. If $V := R^1\pi_* \QQ(\C_\gamma)$ is the standard symplectic local system, we have $R^i\pi_* \QQ(M_{1,0}(\C_\gamma)) = \wedge^i V$ \cite[Proposition 4.9]{SP}, and by Markman's tautological generation result \cite[Theorem 7]{Markman} there is a surjection 
$$A \otimes \bigwedge\nolimits^\bullet(\oplus_{j=2}^r V) \twoheadrightarrow \oplus_i R^i\pi_* \QQ(\widehat{M}_{r,d}(\C_\gamma))$$
\cite[Corollary 4.8]{SP}, where $A$ is a direct sum of trivial local systems. This implies that $\ker(\pi_1(\M) \acts H^1(C_\gamma)) \subset \ker(\pi_1(\M) \acts H^*(M_{r,d}(C_\gamma)))$.
On the other hand, since also
$$R^1\pi_* \QQ(M_{r,d}(\C_\gamma)) \supset R^1\pi_* \QQ(M_{1,0}(\C_\gamma)) \otimes R^0\pi_* \QQ(\widehat{M}_{r,d}(\C_\gamma)) = V \otimes \QQ = V$$
we have the opposite inclusion $\ker(\pi_1(\M) \acts H^1(\C_\gamma)) \supset \ker(\pi_1(\M) \acts H^*(M_{r,d}(\C_\gamma))).$
\end{proof}

Let $C_n \to C$ be the \'etale cover corresponding to the canonical map
$$\pi_1(C) \to H_1(C, \ZZ) \to H_1(C, \ZZ/n)$$
and $\C_n \to \C$ the corresponding relative cover over $\M$.

\begin{corollary} \label{cor: upper}
For $W$ as in Theorem \ref{thm: relative}, we have
$$\ker(\pi_1(W) \acts H^1(C_n)) \subset \ker(\pi_1(W) \acts H^*(\widecheck{M}_{n,L}(C))).$$
\end{corollary}

\begin{proof}
Recall from Theorem \ref{thm: relative} that we have a surjection $\oplus_{\gamma \in \Gamma} H^*(M_{r,d}(C_\gamma)) \twoheadrightarrow H^*(\widecheck{M}_{n,L}(C))$ of (complex) $\pi_1(W)$-representations, and so
\begin{equation} \label{eq: upper LHS 1}
    \ker(\pi_1(W) \acts \oplus_\gamma H^*(M_{r,d}(C_\gamma))) \subset \ker(\pi_1(W) \acts H^*(\widecheck{M}_{n,L}(C))) .
\end{equation}
Next, we have
\begin{multline} \label{eq: upper LHS 2}
    \ker(\pi_1(W) \acts \oplus_\gamma H^*(M_{r,d}(C_\gamma)))) = \cap_\gamma \ker(\pi_1(W) \acts H^*(M_{r,d}(C_\gamma))) \\
    = \cap_\gamma \ker(\pi_1(W) \acts H^1(C_\gamma)))
\end{multline}
by Lemma \ref{lemma: GL}.

Now, for each $\gamma \in \Gamma = H_1(C, \ZZ/n)^*$ we have a covering map $C_n \to C_\gamma$ corresponding to $H_1(C, \ZZ/n) \xrightarrow{\gamma} \CC^*$. This gives a $\hat{\Gamma} = H_1(C, \ZZ/n) = \Gal(C_n/C)$ and $\pi_1(W)$-equivariant identification $H^1(C_\gamma) \cong H^1(C_n)^{\ker \gamma}$. Moreover, the $\pi_1(W)$ and $\hat{\Gamma}$ actions on $H^1(C_n)$ commute, as $\hat{\Gamma}$ is trivialized over $W$. Breaking down $H^1(C_n)$ according to characters of $\hat{\Gamma}$, we
conclude that
$$\ker(\pi_1(W) \acts H^1(C_n))) = \cap_{\gamma \in \Gamma} \ker(\pi_1(W) \acts (H^1(C_n)_\gamma) = \cap_\gamma \ker(\pi_1(W) \acts H^1(C_\gamma)),$$
as for each $\gamma$ we have
$$\ker(\pi_1(W) \acts H^1(C_\gamma)) = \cap_{i=1}^{|\gamma|} \ker(\pi_1(W) \acts H^1(C_n)_{i \gamma}).$$
Combining this with (\ref{eq: upper LHS 1}) and (\ref{eq: upper LHS 2}) above, the result follows.
\end{proof}

\begin{lemma} \label{lemma: lower}
For any $\kappa \ne 0 \in \hat{\Gamma}$ we have
$$\ker(\pi_1(W) \acts H^*(\widecheck{M}_{n,L}(C))^\Gamma \oplus H^*(\widecheck{M}_{n,L}(C))_\kappa) \subset \ker(\pi_1(W) \acts H^1(C_\gamma)_{-\kappa_G}),$$
where $\kappa_G \in \hat{G}_\gamma$ is as defined before Proposition \ref{prop: st = kappa}.
\end{lemma}

\begin{proof}
By the results of the previous section, we have isomorphisms of complex $\pi_1(W)$-representations
$$H^*(\widecheck{M}_{n,L}(C))_\kappa \overset{\ref{cor: last step}, \ref{prop: Gpi quotient}}{\cong} H^*(M_\pi(C_\gamma))_\kappa^{G_\gamma} \overset{\ref{prop: st = kappa}}{\cong} H^*(M_\pi(C_\gamma))^\Gamma_{-\kappa_G}$$
and
$$H^*(M_{r,d}(C_\gamma))_{-\kappa_G} \overset{\ref{prop: step 1}}{\cong} H^*(M_{1,0}(C)) \otimes H^*(M_\pi(C_\gamma))^\Gamma_{-\kappa_G}.$$
Furthermore, from the proof of Lemma \ref{lemma: GL} (applied in the case of $\kappa = 0$) we see that
$$\ker(\pi_1(W) \acts H^*(\widecheck{M}_{n,L}(C))^\Gamma) = \ker(\pi_1(W) \acts H^1(C)) = \ker(\pi_1(W) \acts M_{1,0}(C)),$$
(where the first equality uses the identification $H^*(\widecheck{M}_{n,L}(C))^\Gamma = H^*(\widehat{M}_{n,d}(C))$),
and so
\begin{equation*}
    \ker(\pi_1(W) \acts H^*(\widecheck{M}_{n,L}(C))^\Gamma \oplus H^*(\widecheck{M}_{n,L}(C))_\kappa) \subset \ker(\pi_1(W) \acts H^*(M_{r,d}(C_\gamma))_{-\kappa_G}).
\end{equation*}
Next, note that the $G_\gamma$-equivariant map
$$\det: M_{r,d}(\C_{\gamma, W}) \to \Pic^d_{C_{\gamma, W}/W}, (E, \theta) \mapsto \det E$$
induces an injection of $\pi_1(W)$-representations $H^*(\Pic^d(C_\gamma)) \hookrightarrow H^*(M_{r,d}(C_\gamma))$. (Indeed, for any choice of $(E, \theta) \in M_{r,d}(C_\gamma)$, we can compose the determinant map with
$$\Pic^0(C_\gamma) \to M_{r,d}(C_\gamma), A \mapsto (A \otimes E, 1 \otimes \theta)$$
to get the isogeny $\Pic^0(C_\gamma) \to \Pic^d(C_\gamma), A \mapsto A^{\otimes r} \otimes \det E$.)
We therefore conclude that
$$\ker(\pi_1(W) \acts H^*(M_{r,d}(C_\gamma))_{-\kappa_G}) \subset \ker(\pi_1(W) \acts H^*(\Pic^d(C_\gamma))_{-\kappa_G})$$
$$\subset \ker(\pi_1(W) \acts H^1(\Pic^d(C_\gamma))_{-\kappa_G}).$$

Finally, recall that choosing local trivializations of the $\Pic^0_{\C_{\gamma,W}/W}$-torsor $\Pic^d_{\C_{\gamma, W}/W}$ gives a natural isomorphism of $\pi_1(W)$-representations $H^*(\Pic^d(C_\gamma)) \cong H^*(\Pic^0(C_\gamma))$ \cite[Lemma 1.3.5]{PW}. We note that this isomorphism is also $G_\gamma$-equivariant: choosing any $A \in \Pic^d(C_\gamma)$ to trivialize the torsor, we see that $G_\gamma$ acts on $\Pic^d(C_\gamma)$ via its action on $\Pic^0(C_\gamma)$ multiplied by $\xi \mapsto \xi(A) \otimes A^{-1} \in \Pic^0(C_\gamma)$, and since $\Pic^0(C_\gamma)$ is a connected group, this second part acts trivially on cohomology. Thus
$$\ker(\pi_1(W) \acts H^1(\Pic^d(C_\gamma))_{-\kappa_G}) = \ker(\pi_1(W) \acts H^1(\Pic^0(C_\gamma))_{-\kappa_G})$$
$$= \ker(\pi_1(W) \acts H^1(C_\gamma)_{-\kappa_G}),$$
completing the proof.
\end{proof}

Recall that subgroups $H_1, H_2 \subset G$ are commensurable if $H_1 \cap H_2$ is of finite index in each $H_i$.
\begin{theorem} \label{thm: monodromy}
The subgroups
$$\ker(\pi_1(\M_{g,1}) \acts H^*(\widecheck{M}_{n,L}(C), \ZZ)) \,\,\,\,\textrm{and} \,\,\,\,\ker(\pi_1(\M_{g,1}) \acts H^1(C_n, \ZZ))$$
are commensurable in $\pi_1(\M_{g,1})$.
\end{theorem}

\begin{proof}
First of all, since the torsion subgroup of $H^*(\widecheck{M}_{n,L}(C), \ZZ)$ is finite, it suffices to work with cohomology over $\CC$.
Also, since passing to a finite cover has the effect of restricting to a finite index subgroup of $\pi_1(\M_{g,1})$ and restricting to a Zariski open set induces a surjection on $\pi_1$, it suffices to prove that on $\pi_1(W)$ the two kernels are commensurable.

By Corollary \ref{cor: upper} we have
$$\ker(\pi_1(W) \acts H^1(C_n)) \subset \ker(\pi_1(W) \acts H^*(\widecheck{M}_{n,L}(C))).$$ On the other hand, since $\Gamma$ is trivialized over $W$, we have
$$\ker(\pi_1(W) \acts H^*(\widecheck{M}_{n,L}(C)))  = \bigcap_{\kappa \ne 0 \in \hat{\Gamma}} \ker(\pi_1(W) \acts H^*(\widecheck{M}_{n,L}(C))^\Gamma \oplus H^*(\widecheck{M}_{n,L}(C))_\kappa)$$
\begin{equation} \label{eq: char decomp}
    \subset \bigcap_{\kappa \in \hat{\Gamma}} \ker(\pi_1(W) \acts H^1(C_\gamma)_{-\kappa_G})
\end{equation}
by Lemmas \ref{lemma: GL} (for $\kappa = 0$) and \ref{lemma: lower} (for $\kappa \ne 0$). Considering the splitting
$$H^1(C_n) = \oplus_{\gamma \in \Gamma} H^1(C_n)_\gamma = \oplus_{\gamma \in \Gamma} H^1(C_n/(\ker \gamma))_{\gamma} = \oplus_{\gamma \in \Gamma} H^1(C_\gamma)_\gamma$$
we see that the intersection (\ref{eq: char decomp}) runs over all possible characters. (Recall that we defined $\kappa_G \in \widehat{G}_\gamma \subset \Gamma$ to match $\gamma \in \widehat{\Gamma/\ker(\kappa)} \subset \widehat{\Gamma}$ via the isomorphism of Lemma \ref{lemma: Ggamma = Gamma/K}. Both the cover $C_\gamma \to C$ and this isomorphism are independent of the choice of generator of $\langle \gamma \rangle$, and so as $\kappa$ ranges over $\widehat{\Gamma}$, correspondingly $-\kappa_G$ ranges over $\Gamma$.)
\end{proof}

The $\pi_1(\M_g)$ action on $H_1(C_n)$ is described in \cite{Looijenga}, apart from a few exceptions in genus 2; the problem arises in the argument that the monodromy representations on the different isotypic components of $H_1(C_n)$, each of which is understood individually \cite[Theorem 2.4]{Looijenga}, are (virtually) independent. (The Hodge-theoretic big monodromy result in \cite{LLS}, which only describes algebraic monodromy but holds also for nonabelian Galois groups, only applies in genus $> 2$.) 

\begin{corollary} \label{cor: monodromy}
If either $g > 2$ or $g = 2$ and $(n,6) = 1$, then the algebraic monodromy group of $\pi(\M_{g,1}) \acts H^*(\widecheck{M}_{n,L}(C), \CC)$ is the derived subgroup of $\Sp_{\hat{\Gamma}}(H^1(C_n, \CC))$, i.e.
$$\begin{cases}
\SL_{2(g-1)}(\CC)^{\times (n^{2g} - 2^{2g})/2} \times \Sp_{2(g-1)}(\CC)^{\times (2^{2g} -1)} \times \Sp_{2g}(\CC) & n \equiv 0(2) \\
\SL_{2(g-1)}(\CC)^{\times (n^{2g} - 1)/2} \times \Sp_{2g}(\CC) & n \equiv 1(2).
\end{cases}
$$
If $g = 2$ and $(n,6) \ne 1$, then the algebraic monodromy is contained in this group.
\end{corollary}

\begin{proof}
First, by Theorem \ref{thm: monodromy}, and the fact that the $\pi_1(\M_{g,1})$ action on $H^1(C_n)$ factors through $\pi_1(\M_g)$ (i.e., that the quotient $C_n \to C$ is canonical, and thus extends to a well-defined cover over $\M_g$), it suffices to compute the algebraic monodromy group of $\pi_1(\M_g) \acts H^1(C_n, \CC)$. Assuming that $g$ and $n$ are as described, Looijenga's result is that this is as big as possible, given that the monodromy action is compatible with the intersection form and (after passing to a finite index subgroup) the $\widehat{\Gamma}$ action on $C_n$, and must be semisimple \cite{Hodge2}; more precisely, the image of a finite-index subgroup of $\pi_1(\M_g)$ is an arithmetic subgroup of the commutator subgroup of $\Sp_{\hat{\Gamma}}(H^1(C_n, \QQ))$ \cite[Corollary 2.6]{Looijenga}. By the Borel density theorem, an arithmetic subgroup of this semisimple group is Zariski dense, and so the closure over $\CC$ is $[\Sp_{\hat{\Gamma}}(H^1(C_n, \CC)), \Sp_{\hat{\Gamma}}(H^1(C_n, \CC))] \subset \GL(H^1(C_n, \CC))$. (On the other hand, if $g = 2$ and $(n,6) \ne 1$, the algebraic monodromy may be a proper subgroup.)

This group may be described explicitly as follows: first,
an automorphism of $H^1(C_n, \CC)$ commutes with $\widehat{\Gamma}$ exactly if it fixes all isotypic components $H^1(C_n, \CC)_\gamma$, and since the intersection form is $\widehat{\Gamma}$-invariant, each $H^1(C_n, \CC)_\gamma$ pairs nontrivially exactly with $H^1(C_n, \CC)_{-\gamma}$. In other words, we have
$$\Sp_{\hat{\Gamma}}(H^1(C_n, \CC)) = \prod_{\{\gamma, - \gamma\} \subset \Gamma} \Sp_{\hat{\Gamma}}(H^1(C_n)_\gamma + H^1(C_n)_{-\gamma})$$
where
$$\Sp_{\hat{\Gamma}}(H^1(C_n)_\gamma + H^1(C_n)_{-\gamma}) = \begin{cases}
\GL(H^1(C_n)_\gamma) & \gamma \ne - \gamma \\
\Sp(H^1(C_n)_\gamma) & \gamma = - \gamma.
\end{cases}
$$
(Indeed, in the first case, a $\widehat{\Gamma}$-equivariant transformation will fix $H^1(C_n)_{\pm \gamma}$, which are dual under the intersection pairing, so a symplectic action on the sum is determined by an automorphism on $H^1(C_n)_\gamma$. In the second case, the intersection form restricts nondegenerately to $H^1(C_n)_\gamma$, and any action on this space is $\widehat{\Gamma}$-invariant.) Finally, we note that we have $[\GL_n(\CC), \GL_n(\CC)] = \SL_n(\CC)$ and $[\Sp_{2g}(\CC), \Sp_{2g}(\CC)] = \Sp_{2g}(\CC)$ and that $\dim H^1(C_n)_\gamma = 2(g-1)$ for $\gamma \ne 0$ and $2g$ for $\gamma = 0$ (by \cite[Proposition 4.2]{Looijenga} and the fact that $H^1(C_n)_\gamma \cong H^1(C_\gamma)_\gamma$).
\end{proof}

\subsection{Back to the character variety}
It now remains to show that the $\pi_1(\M_{g,1}) = \Mod(C \setminus p)$ actions on $H^*(\widecheck{M}_{n,L}(C))$ and $H^*(M^{\SL})$ agree. This is presumably well-known to experts, but it is nevertheless not immediate from Simpson's original nonabelian Hodge correspondence, which only describes the case $d = 0$.

Recall that $\pi_1(\M_{g,1}) = \Mod(C \setminus p)$ acts on
$$M^{\SL} = \{\phi \in \Hom(\pi_1(C \setminus p), \SL_n(\CC)): \phi(\gamma_p) = e^{2\pi i d/n} \Id_n\}/\PGL_n(\CC)$$
(where $\gamma_p$ is a counterclockwise loop around the puncture $p$)
via its outer action on $\pi_1(C \setminus p)$. (See \cite[Definition 2.3]{SP} for more discussion.) 
We also consider the ``degree $d$" $\PGL_n$ character variety
$$M^{\PGL} := M^{\SL}/\Hom(\pi_1(C \setminus p), \mu_n),$$
where the action of $\Hom(\pi_1(C \setminus p), \mu_n)$ is defined by the action of $\mu_n$ on $\SL_n$ by scalar multiplication. The $\pi_1(\M_{g,1})$ action on $H^*(M^G)$ can be described as a monodromy action on the ``local system of varieties" $M^G(\C_{g,1}) := \mc T_{g,1} \times_{\pi_1(\M_{g,1})} M^G$ over $\M_{g,1}$, which is a quotient of the Teichm\"uller space $\mc T_{g,1}$ by the fundamental group $\pi_1(\M_{g,1})$.

As mentioned in the proof of Lemma \ref{lemma: GL}, there is also a moduli space $\widehat{M}_{n,d}(C) := \widecheck{M}_{n,L}(C)/\Pic^0(C)[n]$ of ``degree $d$" $\PGL$ Higgs bundles, where $\Pic^0(C)[n] = \Gamma$ acts by tensor product. Similarly, there is a relative version $\widehat{M}_{n,d}(\C_{g,1})$ over $\M_{g,1}$, and it follows from Simpson's degree 0 relative nonabelian Hodge correspondence \cite[Theorem 9.11, Lemma 9.4]{Simpson-reps2} that there is a canonical homeomorphism of $M^{\PGL}(\C_{g,1})$ and $\widehat{M}_{n,d}(\C_{g,1})$ over $\M_{g,1}$. (See \cite[\S 3]{SP} for a discussion of the relative $\PGL$ nonabelian Hodge correspondence in this form.)

\begin{prop} \label{prop: relative NAH}
The two monodromy representations of $\pi_1(\M_{g,1})$ on $H^*(\widecheck{M}_{n,L}(C)) \cong H^*(M^{\SL})$ agree on a finite index subgroup.
\end{prop}

\begin{proof}
As described, e.g., in \cite[\S 2]{HT}, over each point $(C,p)$ of $\M_{g,1}$ there is a (not necessarily canonical) nonabelian Hodge homeomorphism $\widecheck{M}_{n,L}(C) \cong M^{\SL}(C \setminus p)$ descending to the canonical homeomorphism $\widehat{M}_{n,d}(C) \cong M^{\PGL}(C \setminus p)$ under the quotient map induced by $\SL_n \to \PGL_n$ on each side of the correspondence.

Now, since the $\Gamma$ action on $\widecheck{M}_{n,L}(C)$ is generically free and thus $\widecheck{M}_{n,L}(C) \to \widehat{M}_{n,d}(C) = \widecheck{M}_{n,L}(C)/\Gamma$ is \'etale over a Zariski open set (whose preimage is dense in the Euclidean topology on $\widecheck{M}_{n,L}(C)$), we see that isomorphisms of $\widecheck{M}_{n,L}(C)$ over the identity map of $\widehat{M}_{n,d}$ are determined by their restrictions to this open subset and thus indexed by the covering group $\Gamma$. Then there is an $n^{2g}$-fold cover $S$ of $\M_{g,1}$ whose fiber over $(C,p)$ consists of the fiberwise homeomorphisms $\widecheck{M}_{n,L}(C) \cong M^{\SL}(C \setminus p)$ over the canonical homeomorphism $\widehat{M}_{n,d}(C) \cong M^{\PGL}(C \setminus p)$ given by Simpson's nonabelian Hodge correspondence, and by construction there is a relative homeomorphism of $\widecheck{M}_{n,\ms L}(\C_S)$ and $M^{\SL}(\C_S)$ over $S$; in particular, the monodromy representations of $\pi_1(S)$ on $H^*(\widecheck{M}_{n,L}(C))$ and $H^*(M^{\SL})$ agree.
\end{proof}

\begin{remark}
The nonabelian Hodge correspondence of \cite{HT} is not a priori canonical because, as described in \cite{HT-generators} in the case of $\GL$, one passes from the moduli space of connections with constant central curvature $di\omega \Id_n$ on $C$ (canonically identified with $M_{n,d}(C)$) to the space of flat connections on $C \setminus p$ with holonomy $e^{2\pi i d/n} \Id_n$ around the puncture (canonically identified with $M^{\GL}(C \setminus p)$) by taking holonomies around an arbitrarily selected set of loops corresponding to generators $\pi_1(C \setminus p)$. However, since changing the choice of loop in a homotopy class has the effect of scaling the holonomy by $\CC^*$, we do recover a canonical correspondence of $\PGL$ spaces.

In fact, there should be a canonical degree $d$ $\SL$ nonabelian Hodge correspondence over $\M_{g,1}$ given by Simpson's parabolic nonabelian Hodge correspondence on the punctured curve $C \setminus p$ \cite{Simpson-noncompact}, or more or less equivalently, on the root stack $C(\frac{p}{n})$ \cite{Simplicial}. However, since our results only hold on a finite cover of $\M_{g,1}$ in any case, and since there are complications involved in formulating these correspondences for reductive groups other than $\GL_n$ (though see \cite{huang2023tameparahoricnonabelianhodge} and \cite{huang2025modulispacesfilteredglocal} for complete results on parabolic bundles on curves), we do not pursue this any further.
\end{remark}

In view of the previous proposition, Theorem \ref{thm: intro} and Corollary \ref{cor: intro} are equivalent to their Higgs bundle counterparts, Theorem \ref{thm: monodromy} and Corollary \ref{cor: monodromy}.

\printbibliography
\Addresses

\end{document}